\RequirePackage{fix-cm}
\documentclass[smallextended]{svjour3}       
\smartqed  
\usepackage{graphicx}
\usepackage[numbers]{natbib}

\setcitestyle{square}

\usepackage{amsmath} 

\usepackage{amssymb} 

\newtheorem{prop}{Proposition}

\newtheorem{thm}{Theorem}

\newtheorem{Lemma}{Lemma}

\def\tr{\operatorname{tr}}
\def\Ad{\operatorname{Ad}}
\def\Gr{\operatorname{Gr}}

\begin{document}

\title{Monotonicity on homogeneous spaces
\thanks{This work was funded by the Engineering and Physical Sciences Research Council (EPSRC) of the United Kingdom, as well as the European Research Council under the Advanced ERC Grant Agreement Switchlet n.670645.}
}


\author{Cyrus Mostajeran        \and
        Rodolphe Sepulchre 
}


\institute{C. Mostajeran \at
              Department of Engineering, University of Cambridge, Cambridge, UK \\
              \email{csm54@cam.ac.uk}           
           \and
           R. Sepulchre \at
              \email{r.sepulchre@eng.cam.ac.uk}  
}

\date{December 20, 2018}

\maketitle

\begin{abstract}
This paper presents a formulation of the notion of monotonicity on homogeneous spaces. We review the general theory of invariant cone fields on homogeneous spaces and provide a list of examples involving spaces that arise in applications in information engineering and applied mathematics. Invariant cone fields associate a cone with the tangent space at each point in a way that is invariant with respect to the group actions that define the homogeneous space. We argue that invariance of conal structures induces orders that are tractable for use in analysis and propose invariant differential positivity as 
a natural generalization of monotonicity on such spaces.

\end{abstract}

\section{Introduction}

Monotonicity is the property of dynamical systems or maps that preserve a partial order, which is defined as a binary relation that is reflexive, antisymmetric, and transitive. That is, a monotone dynamical system is characterized by the property that any two points that are ordered at one instant in time will remain ordered at all subsequent times as the system evolves with the flow. Monotone flows and their discrete-time analogues, order-preserving maps, play an important role in the theory of dynamical systems and find applications to many biological, physical, chemical, and economic models \cite{Luenberger1979,Farina2000}. These systems are closely related to linear dynamical systems with input and output channels, where the monotonicity of a nonnegative input is preserved by the output  \cite{Ohta1984,Anderson1996,Grussler2014,Altafini2016,Grussler2017}. Recently, this type of input-output preserving system has been further extended to the notion of unimodality \cite{Grussler2018}.

One area of application of monotonicity is the theory of consensus algorithms \cite{Moreau2004,Olfati-Saber2006,Olfati-Saber2007,Jadbabaie2003,Sepulchre2010,Sepulchre2011}, where one is interested in designing and analyzing consensus protocols that define the interactions between a collection of agents exchanging information about their relative states via a communication network with the aim of achieving collective behavior. Monotone systems also arise naturally in many areas of biology \cite{Enciso2005,Angeli2008,Angeli2012}. A classical example of a monotone system arising from biology is described by a Kolmogorov model of interacting species where an increase in any population causes an increase in the growth rate of all other populations. Such systems are said to be cooperative \cite{Smith2008}. Monotone systems also arise in areas of biology other than population dynamics. For instance, see \cite{Smith2003} for an example concerning the dynamics of viral infections. Furthermore, monotone subsystems are often found as components of larger networks due to their robust dynamical stability and predictability of responses to perturbations. The decomposition of networks into monotone subsystems and the study of their interconnections using tools from control theory have also proven to be insightful \cite{Angeli2003,Angeli2004,de2007}.

This paper addresses the question of how to define monotonicity on a homogeneous manifold. The notion of order plays a defining role in monotonicity theory. In linear spaces, it is well-known that orders are intimately connected with the theory of pointed solid convex cones. In this paper, a solid convex cone $\mathcal{K}$ is said to be pointed if $\mathcal{K}\cap-\mathcal{K}=\{0\}$. Every such cone $\mathcal{K}$ induces a partial order $\leq$ in a vector space, whereby $a\leq b$ if  $b-a$ lies in $\mathcal{K}$. The simplest example is provided by the positive orthant $\mathbb{R}^n_+$ in $\mathbb{R}^n$ consisting of vectors with nonnegative entries, which induces the standard vector order based on pairwise comparisons of vector entries.  It is natural to generalize this approach by defining a field of cones on a manifold, whereby a cone is associated with the tangent space at each point on the manifold. A \emph{conal curve} of a cone field is defined as a piecewise smooth curve whose tangent vector lies in the cone at every point along the curve wherever it exists. Cone fields induce the notion of conal orders, whereby a pair of points are said to be ordered if the first point can be joined to the second point by a conal curve. Conal orders locally define partial orders on a manifold. Whether the local partial order can be extended globally depends on the structure of the cone field and the underlying space. 

In most applications in applied mathematics and engineering, we are interested in problems that are formulated on spaces with special geometries such as homogeneous spaces. These are manifolds that admit a transitive Lie group action and thus provide a way of systematically generating mathematical structures over the tangent bundle using constructs defined at a single point. This provides a methodology to incorporate the symmetries of the space in any additional structures that are endowed to the space for analysis and design purposes. The classical example of such a construction is that of a homogeneous Riemannian metric, which is entirely determined by the metric at a point. In a similar spirit, we define and characterize invariant cone fields on homogeneous spaces. In doing so we closely review elements of the general theory of homogeneous cone fields as outlined in the important work by Hilgert et al. in \cite{Hilgert1989,Hilgert2006} and Neeb in \cite{Neeb1991}. We then present a number of examples of homogeneous spaces that arise in a variety of applications in information geometry, computational science and engineering, including Grassmann manifolds and spaces of symmetric positive definite matrices, and consider the existence of invariant cone fields on these spaces. A key theme of the paper is that geometric invariance yields `tractability' in the analysis of orders and related concepts, which otherwise may appear daunting. In particular, we show that conality of geodesics on globally orderable Riemannian homogeneous spaces can be used to determine order relations between points on such spaces.

Cone fields and conal curves provide a local or differential way of thinking about order relations, which can be viewed as corresponding global concepts. Monotonicity itself is a global concept in the sense that it is classically defined in relation to some partial order. In extending any concept defined on vector spaces to manifolds, it is natural to seek the differential characterization of the property, which in turn will often provide a route for generalization to the nonlinear manifold setting. The local property that is equivalent to monotonicity in $\mathbb{R}^n$ with respect to a partial order defined by a constant cone field is differential positivity \cite{Forni2015}. We propose \emph{invariant differential positivity} (i.e., differential positivity with respect to an invariant cone field) as a generalization of monotonicity to homogeneous spaces. We will show that invariant differential positivity is indeed equivalent to monotonicity when the cone field induces a global partial order. Furthermore, we discuss how the property remains useful in cases where the order is not a global partial order.
Invariant differential positivity can be a powerful analytic tool for the study of monotonicity in a variety of contexts, including the theory of consensus of oscillators \cite{Mostajeran2016,Mostajeran2017a}, nonlinear dynamical systems \cite{Forni2014a,Forni2015a,Forni2015}, and matrix monotone functions \cite{Lowner1934,Bhatia,Mostajeran2017b}. In this paper, we specify what we mean by invariant differential positivity on homogeneous spaces, including with respect to cones of rank $k$ \cite{sanchez2009,sanchez2010,Fusco1991,Mostajeran2017d}, which are generalizations of cones to structures that are closed and invariant under scaling by all real numbers. We also consider the strong implications that invariant differential positivity can have for the asymptotic behavior of dynamical systems.

\section{Homogeneous spaces} \label{2}

A \emph{left action} of a Lie group $G$ on a manifold $\mathcal{M}$ is a smooth map $\Phi:G\times \mathcal{M}\rightarrow \mathcal{M}$ satisfying $\Phi(e,x)=x$ and $\Phi(g_1g_2,x)=\Phi(g_1,\Phi(g_2,x))$ for all $g_1,g_2\in G$, $x\in \mathcal{M}$, where $e$ is the identity element in $G$.
Note that for a given $g\in G$, the map $x\mapsto\Phi(g,x)$ is a diffeomorphism of $\mathcal{M}$. The group $G$ is referred to as a transformation group of the manifold $\mathcal{M}$. We will use $\Phi(g,x)$ and $g\cdot x$ interchangeably in this paper.
A homogeneous space is defined as a manifold $\mathcal{M}$ on which a Lie group $G$ acts transitively.

\begin{definition}
A smooth manifold $\mathcal{M}$
is said to be a homogeneous space if there exists a Lie group $G$ acting on $\mathcal{M}$ such that for all $x_1,x_2\in \mathcal{M}$, there exists $g\in G$ such that $g\cdot x_1=x_2$.
\end{definition}

Homogeneous spaces are closely connected to \emph{coset manifolds}. For a given Lie group $G$ and a closed subgroup $H$, consider the set $G/H:=\{g H:g\in G\}$ of left cosets of $H$ in $G$. The set $G/H$ is the set of equivalence classes for the equivalence relation $\sim$ on $G$ defined by 
\begin{equation}
g_1\sim g_2 \quad \Longleftrightarrow \quad \exists \, h\in H :\, g_1= g_2 h.
\end{equation}
The set $G/H$ can be made into a manifold in a unique way if we require that the projection map $\pi:G\rightarrow G/H$, $\pi(g):=gH$ be a submersion; i.e., if we require that the differential map $d\pi\vert_g$ is surjective for each $g\in G$. For each $a\in G$, define the left translation $\tau_a:G/H\rightarrow G/H$ by $\tau_a(gH):=agH$. Note that the left translations $\tau_g$ are related to the left translations $L_g$ on the Lie group $G$ by
$\pi\circ L_g=\tau_g\circ \pi$, for each $g\in G$. The left translations $\tau_a$ define a transitive action on $G/H$ given by $\Phi(a,gH):=\tau_a(gH)=agH$. Thus, all coset manifolds of the form $G/H$ are homogeneous spaces. Indeed,  the converse is also true. That is, any homogeneous manifold $\mathcal{M}$ with a transitive group action $G\times\mathcal{M}\rightarrow\mathcal{M}$ can be expressed as a suitable coset manifold $G/H$. To see this, we first define the isotropy group $G_x$ at a point $x\in \mathcal{M}$ to be the set $G_x:=\{g\in G:g\cdot x =x\}$. That is, the isotropy group $G_x$ consists of all elements in the transformation group $G$ that keep $x$ fixed. Fix a point $o\in\mathcal{M}$ and note that $H:=G_o$ forms a closed subgroup of $G$. The natural map $\iota:G/H\rightarrow\mathcal{M}$ defined by $\iota(gH)=g\cdot o$ is a diffeomorphism, so that $\mathcal{M}\cong G/H$ as smooth manifolds. Furthermore, it can be shown that $\dim\mathcal{M}=\dim G - \dim H$ \cite{Arv2003}.

\subsection{Reductive homogeneous spaces}

Let $\mathcal{M}=G/H$ be a homogeneous space and consider the natural projection $\pi:G\rightarrow G/H$, $\pi(g)=gH$. The differential $d\pi\vert_e:\mathfrak{g}\rightarrow T_{o}(G/H)$, where $o=\pi(e)=eH$, is given by
\begin{equation} \label{diff proj}
d\pi\vert_e X=\frac{d}{dt}(\pi\circ\exp tX)\Big |_{t=0}=\frac{d}{dt}\left((\exp tX)H\right)\Big |_{t=0},
\end{equation}
for $X\in\mathfrak{g}$. As
the map $d\pi\vert_e:T_eG\rightarrow T_o\mathcal{M}$ is a vector space homomorphism, we have $T_eG/(\ker d\pi\vert_e)\cong \mathrm{im}\,d\pi\vert_e$. It follows from (\ref{diff proj}) that $\ker d\pi\vert_e=\mathfrak{h}$, where $\mathfrak{h}$ is the Lie algebra of $H$. Thus, we have the canonical isomorphism
\begin{equation}  \label{4.2}
\mathfrak{g/h}\cong T_o(G/H) = T_o\mathcal{M},
\end{equation}
where $\mathfrak{g/h}$ is the set of cosets $X+\mathfrak{h}=\{X+Y|Y\in\mathfrak{h}\}$ for $X\in\mathfrak{g}$.
\begin{definition}
A homogeneous space $\mathcal{M}=G/H$ is said to be \emph{reductive} if there exists a subspace $\mathfrak{m}$ of $\mathfrak{g}$ such that $\mathfrak{g}=\mathfrak{h\oplus m}$ and 
$
\mathrm{Ad}(h)\mathfrak{m}\subseteq\mathfrak{m}$,  for all $h\in H$.
 \end{definition}
The $\mathrm{Ad}(H)$-invariance condition $\mathrm{Ad}(h)\mathfrak{m}\subseteq\mathfrak{m}$ implies $[\mathfrak{h},\mathfrak{m}]\subseteq\mathfrak{m}$. For a reductive homogeneous space $G/H$, the canonical isomorphism (\ref{4.2}) reduces to
$
\mathfrak{m}\cong T_o(G/H)$.
Note that if the Lie group $G$ is compact, then the homogeneous space $\mathcal{M}=G/H$ is reductive since $\mathfrak{g=h\oplus m}$, where $\mathfrak{m}:=\mathfrak{h}^{\perp}$ with respect to an $\mathrm{Ad}$-invariant inner product on $\mathfrak{g}$. Moreover, we note that the Killing form $B:\mathfrak{g\times g}\rightarrow\mathbb{R}$, 
$
B(X,Y)=\mathrm{tr}(\mathrm{ad}X\circ\mathrm{ad}Y)
$
 of a Lie group $G$ is always an $\mathrm{Ad}$-invariant symmetric bilinear form on $\mathfrak{g}$. Thus, if $G$ is compact and semisimple so that $-B$ is positive definite, then the Killing form defines a bi-invariant metric on $G$ given by $\langle \cdot,\cdot\rangle=-B(\cdot,\cdot)$ \cite{Arv2003}.

\subsection{Symmetric spaces}

Symmetric spaces constitute an important class of homogeneous spaces that includes many of the spaces that are of interest in applications and discussed in this paper.
A connected Riemannian manifold $\mathcal{M}$ is said to be a \emph{symmetric space} if for each $p\in\mathcal{M}$, there exists an isometry $j_p:\mathcal{M}\rightarrow \mathcal{M}$, such that 
\begin{equation}
j_p(p)=p \quad \mathrm{and} \quad dj_p\vert_p=-\mathrm{Id}_p,
\end{equation}
where $\mathrm{Id}_p$ is the identity map on $T_p\mathcal{M}$.
The map $j_p$ has the property that it ``reverses" the geodesics that pass through $p\in\mathcal{M}$, in the sense that if $\gamma_v:(-\epsilon,\epsilon)\rightarrow\mathcal{M}$ is the unique geodesic through $p$ with
$\gamma_v(0)=p$ and $\gamma'_v(0)=v$, then $j_p(\gamma_v(t))=\gamma_v(-t)$. The Euclidean space $\mathbb{R}^n$ is clearly symmetric. A less trivial example is the $n$-sphere $\mathbb{S}^n$ embedded in $\mathbb{R}^{n+1}$, where the symmetry at the north pole $p=(1,0,\cdot\cdot\cdot,0)$ is given by $j_p(x_1,x_2,\ldots,x_{n+1})=(x_1,-x_2,\ldots,-x_{n+1})$. 
A symmetric Riemannian manifold $\mathcal{M}$ is a homogeneous space $G/H$, where $G=I(\mathcal{M})$ is the isometry group of $\mathcal{M}$ acting transitively on $\mathcal{M}$ and $H$ is the isotropy subgroup of a point $o\in\mathcal{M}$.

Denote the symmetry of the symmetric space $\mathcal{M}=G/H$ at $o=eH$ by $j$. Now for each $g\in G$, define the map $\sigma(g):\mathcal{M}\rightarrow\mathcal{M}$ by $\sigma(g)=j\circ g\circ j$. Since $\sigma(g)$ is an isometry of $\mathcal{M}$, it lies in $G$ and thus we can define an automorphism $\sigma:G\rightarrow G$ by $g\mapsto \sigma(g)=j\circ g\circ j^{-1}$ as $j^2=\mathrm{Id}$. Setting $G_\sigma=\{g\in G:\sigma(g)=g\}$ to be the fixed points of $\sigma$ and $G_{\sigma}^o$ its connected component, one can show that $\sigma^2=\mathrm{Id}_G$ and $G_{\sigma}$ is a closed subgroup of $G$ that satisfies $G_{\sigma}^o\subseteq H \subseteq G_{\sigma}$. The map $\sigma$ is sometimes referred to as the \emph{involution} map associated with the symmetric space $G/H$. Every symmetric space $G/H$ with involution $\sigma$ is a reductive homogeneous space with reductive decomposition $\mathfrak{g}=\mathfrak{h\oplus m}$, where
\begin{equation}
\mathfrak{h}=\{X\in\mathfrak{g}:d\sigma\vert_e X=X\} \quad \mathrm{and} \quad \mathfrak{m}=\{X\in\mathfrak{g}:d\sigma\vert_e X=-X\},
\end{equation}
and $\mathrm{Ad}(H)\mathfrak{m}\subseteq\mathfrak{m}$ \cite{Arv2003}.

\section{Invariant cone fields on homogeneous spaces}  \label{3}

\subsection{Homogeneous cone fields}

A wedge is a closed and convex subset of a vector space that is closed under multiplication by nonnegative scalars \cite{Hilgert1989}.
A wedge field  $W_{\mathcal{M}}$ on a manifold $\mathcal{M}$  smoothly assigns to each point $x\in\mathcal{M}$ a wedge $W_{\mathcal{M}}(x)$ in the tangent space $T_x\mathcal{M}$. 

\begin{definition} \label{def 6}
Let $\Phi:G\times\mathcal{M}\rightarrow\mathcal{M}$ be any left group action on $\mathcal{M}$ such that each of the maps $\tau_g:\mathcal{M}\rightarrow\mathcal{M}$ defined by $\tau_g(x):=\Phi(g,x)=g\cdot x$ forms a diffeomorphism of $\mathcal{M}$. Then a wedge field $W_{\mathcal{M}}$ is said to be $G$-\emph{invariant} if
\begin{equation} \label{6.4}
d\tau_g\big\vert_x\left(W_{\mathcal{M}}(x)\right) = W_{\mathcal{M}}\left(g\cdot x\right),
\end{equation}
for all $g\in G$ and $x\in \mathcal{M}$.
\end{definition}

We now specialize to the case where the group action is transitive so that $\mathcal{M}$ is a homogeneous space.
A \emph{homogeneous cone field} on a homogeneous space $\mathcal{M}=G/H$ of a connected Lie group $G$ assigns to each point $x\in\mathcal{M}$ a cone $\mathcal{K}_{\mathcal{M}}(x)$ in the tangent space $T_x\mathcal{M}$, such that the cone field is invariant under the action of $G$ on $\mathcal{M}$. Recall that elements of $\mathcal{M}$ can be identified with cosets $gH:=\{gh:h\in H\}$ in $G/H$ and let $o:=eH$ denote the base-point in $\mathcal{M}$. The left translations $\tau_g:\mathcal{M}\rightarrow\mathcal{M}$ are defined by $\tau_g(x)=g\cdot x$ for all $g\in G$ and $x\in\mathcal{M}$. Let $\mathfrak{g}=T_eG$ and $\mathfrak{h}=T_eH$ denote the Lie algebras of $G$ and $H$, respectively. 
The canonical projection $\pi:G\rightarrow\mathcal{M}$ defined by $\pi(g)=gH$ induces a linear surjection $d\pi\vert_e:\mathfrak{g}\rightarrow T_o\mathcal{M}$ with $\ker d\pi\vert_e=\mathfrak{h}$, so that we obtain the isomorphism $\mathfrak{g/h}\cong T_o\mathcal{M}$ given by
\begin{equation}
X+\mathfrak{h}\mapsto d\pi\vert_e X.
\end{equation}
Thus, the surjection $d\pi\vert_e$ is identified with the quotient map $p:\mathfrak{g\rightarrow g/h}$, where $p(X)=X+\mathfrak{h}$. 

Recall that $H$ is the isotropy subgroup of $G$ acting on $\mathcal{M}$ at $o$. That is, for each $h\in H$, we have $\tau_h(o)=o$. Thus, we obtain a vector space isomorphism $d\tau_h\vert_o:T_o\mathcal{M}\rightarrow T_o\mathcal{M}$ and a representation $\eta:H\rightarrow\mathrm{Aut}\left(T_o\mathcal{M}\right)$ of $H$ given by
$
\eta (h)=d\tau_h\big\vert_o$.
Under the identification of $T_o\mathcal{M}$ with $\mathfrak{g/h}$, we have:
\begin{equation}
\eta(h)\left(X+\mathfrak{h}\right)=\mathrm{Ad}(h)\left(X\right)+\mathfrak{h},
\end{equation}
for all $h\in H$ and $X\in\mathfrak{g}$. 

If we seek to describe invariant cone fields on $\mathcal{M}=G/H$ in terms of wedges defined in the Lie algebra $\mathfrak{g}$ of the total space $G$, then there are some consistency requirements that must be satisfied. In particular, any cone $\mathcal{K}$ in $T_o\mathcal{M}=\mathfrak{g/h}$ arising as the projection of a wedge $W$ in $\mathfrak{g}$ must be invariant under the group $\eta(H)$, since otherwise 
we would have different cones at
$o$ depending on the choice of representative $h\in H$ in $o=eH=\pi(h)$. 

\begin{Lemma}
Let $H$ be a closed subgroup of $G$ and $W$ a wedge in $\mathfrak{g}$ with edge $W\cap-W=\mathfrak{h}$. If
\begin{equation}
\mathrm{Ad}(h)\left(W\right)=W, \quad \forall h\in H,
\end{equation}
then the associated pointed cone $\mathcal{K}=p(W)$ in $T_o\mathcal{M}=\mathfrak{g/h}$ is invariant under the group $\eta(H)$.
\end{Lemma}

Note that for a cone $\mathcal{K}$ in $T_o\mathcal{M}$ that is invariant under $\eta(H)$, the $G$-invariant cone field $\mathcal{K}_{\mathcal{M}}$ given by
\begin{equation}
\mathcal{K}_{\mathcal{M}}\left(gH\right):=d\tau_g\big\vert_o\mathcal{K},
\end{equation}
is well-defined. That is, for all $g,g'\in G$ corresponding to the same point $x\in\mathcal{M}$ (i.e. for all $g,g'\in G$ satisfying $\pi(g)=\pi(g')$), we have
\begin{equation}
d\tau_g\big\vert_o \mathcal{K} =d\tau_{g'}\big\vert_o\mathcal{K}.
\end{equation}
To see this, note that $\pi(g)=\pi(g')$ precisely if there exists $h\in H$ such that $g'=gh$. Thus, we have
$
d\tau_{g'}\vert_o=d\tau_g\vert_o\circ d\tau_h\vert_o=d\tau_g\vert_o\circ\eta(h)$,
 whence the result follows from the $\eta(H)$-invariance of $\mathcal{K}$. The following theorem from \cite{Hilgert1989} describes the geometry of homogeneous cone fields on $\mathcal{M}=G/H$. See figure \ref{homogeneous cone field}.
 
 \begin{figure} 
\centering
\includegraphics[width=0.55\linewidth]{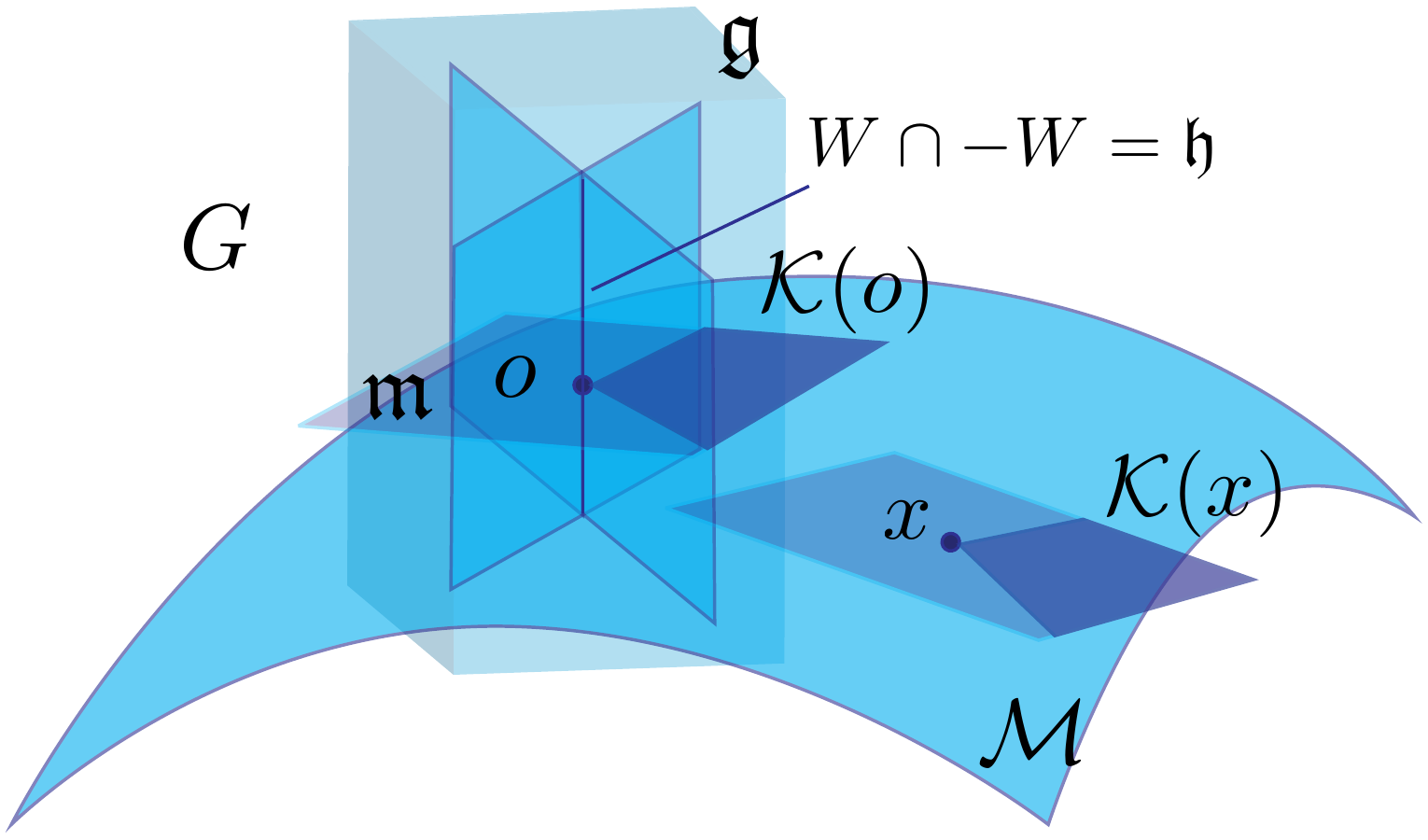}
  \caption{A homogeneous cone field $\mathcal{K}_{\mathcal{M}}$ on $\mathcal{M}=G/H$ arising as the projection of an invariant wedge field $W_G$ on $G$ generated by an $\Ad_{H}$-invariant wedge $W\subset \mathfrak{g}$ that satisfies $W\cap -W=\mathfrak{h}$.
  }
  \label{homogeneous cone field}
\end{figure}

\begin{thm}  \label{cone thm}
Let $H$ be a closed subgroup of a Lie group $G$ and $W$ a wedge in $\mathfrak{g}$ such that
(i) $W\cap-W=\mathfrak{h}$, and
(ii) $\mathrm{Ad}(H)(W)=W$.
Define 
$W_{G}$ and $\mathcal{K}_{\mathcal{M}}$ by
\begin{equation}
W_G(g):=dL_g\big\vert_e\,W, \quad
\mathcal{K}_{\mathcal{M}}(x):=d\tau_g\big\vert_o\,\mathcal{K},
\end{equation}
where $\mathcal{M}:=G/H$, $e$ is the identity element in $G$, $o=eH$ is the base-point in $\mathcal{M}$, and
$\mathcal{K}$ is the pointed cone in $T_o\mathcal{M}$ obtained as the projection of $W$ onto $\mathfrak{g/h}$.
Then, $\mathcal{K}_G$ is an invariant wedge field on $G$ and $\mathcal{K}_{\mathcal{M}}$ is a well-defined homogeneous or $G$-invariant cone field on $\mathcal{M}$. Moreover, for each $g\in G$,
\begin{equation}
d\pi\big\vert_g\left(W_G\right)=\mathcal{K}_{\mathcal{M}}\left(\pi(g)\right),
\end{equation}
where $\pi:G\rightarrow \mathcal{M}$ is the canonical projection $\pi(g)=gH$. 
\end{thm}

\subsection{Examples} \label{examples}

\subsubsection{Lie groups}

Any Lie group $G$ is itself a homogeneous space in at least two ways. First, it can be expressed as $G\cong G/\{e\}$. Alternatively, one can write $G\cong G\times G/G$, where $G\times G$ acts on $G$ by left and right translations and the isotropy subgroup is $G$ diagonally embedded in $G\times G$. Invariant cone fields can be defined on a Lie group as a homogeneous space using left-translation. Given a cone $\mathcal{K}$ in $\mathfrak{g}$, the corresponding left-invariant cone field $\mathcal{K}_G$ is given by
\begin{equation}
\mathcal{K}_G(g)=dL_g\big\vert_e\mathcal{K}, 
\end{equation}
for all $g \in G$.

\subsubsection{The quotient of the Heisenberg group by its center, $G/Z(G)$}

The Heisenberg group $G$ is a Lie group that arises in various fields including representation theory, sub-Riemannian geometry and quantum mechanics.  It can be defined as the group of $3\times 3$ upper triangular matrices with diagonal elements equal to 1 and group operation given by matrix multiplication. The Lie algebra $\mathfrak{g}$ can be represented as the set of stricly upper triangular $3\times 3$ matrices. That is,
\begin{equation}
G=\Bigg\{\begin{pmatrix} 1 & a & c \\ 0 & 1 & b \\ 0 & 0& 1 \end{pmatrix}: a,b,c\in\mathbb{R}\Bigg\}, \quad
\mathfrak{g}=\Bigg\{\begin{pmatrix} 0 & \alpha & \gamma \\ 0 & 0 & \beta \\ 0 & 0& 0 \end{pmatrix}: \alpha,\beta,\gamma\in\mathbb{R}\Bigg\}.
\end{equation}
The center $Z(G)$ of $G$ is defined as the set $Z(G)=\{z\in G: zg=gz, \; \forall g\in G\}$ and forms a subgroup $H$ of $G$ with Lie algebra $\mathfrak{h}$ given by
\begin{equation}
H:=Z(G)=\Bigg\{\begin{pmatrix} 1 & 0 & c \\ 0 & 1 & 0 \\ 0 & 0& 1 \end{pmatrix}: c\in\mathbb{R}\Bigg\}, \quad \mathfrak{h}=\Bigg\{\begin{pmatrix} 0 & 0 & \gamma \\ 0 & 0 & 0 \\ 0 & 0& 0 \end{pmatrix}: \gamma\in\mathbb{R}\Bigg\}.
\end{equation}
The quotient manifold $G/H$ defines a homogeneous space of dimension 2. To construct an invariant cone field on $G/H$, we first look for a wedge $W$ in $\mathfrak{g}$ that satisfies condition $(i)$ of Theorem \ref{cone thm}; i.e., $W\cap-W=\mathfrak{h}$. This is achieved precisely if $W$ is of the form
\begin{equation}  \label{Heisenberg wedge}
W=\{(\alpha,\beta,\gamma): (\alpha,\beta)\in\mathcal{K}\subset\mathbb{R}^2, \gamma\in\mathbb{R}\},
\end{equation}
where $\mathcal{K}$ is any pointed convex solid cone in $\mathbb{R}^2$. For condition $(ii)$ of Theorem \ref{cone thm}, we consider $\Ad(H)W$. Now since
\begin{equation}
\begin{pmatrix} 1 & 0 & c \\ 0 & 1 & 0 \\ 0 & 0& 1 \end{pmatrix}\begin{pmatrix} 0 & \alpha & \gamma \\ 0 & 0 & \beta \\ 0 & 0& 0 \end{pmatrix}\begin{pmatrix} 1 & 0 & c \\ 0 & 1 & 0 \\ 0 & 0& 1 \end{pmatrix}^{-1}=\begin{pmatrix} 0 & \alpha & \gamma \\ 0 & 0 & \beta \\ 0 & 0& 0 \end{pmatrix},
\end{equation}
for all $c\in\mathbb{R}$ and $(\alpha,\beta,\gamma)\in\mathbb{R}^3$, $\Ad(H)(W)=W$ trivially holds for any wedge $W$ in $\mathfrak{g}$. Therefore, any wedge $W$ of the form (\ref{Heisenberg wedge}) uniquely defines an invariant cone field $\mathcal{K}_{\mathcal{M}}$ on $\mathcal{M}=G/H$.

\subsubsection{The $n$-spheres and Grassmannians}

The $n$-sphere $\mathbb{S}^n$ can be viewed as a homogeneous space 
with a transitive $SO(n+1)$ action, since any two points on $\mathbb{S}^n$ embedded in $\mathbb{R}^{n+1}$ are related by a rotation.
We fix the point $o=(1,0,\ldots, 0)\in\mathbb{S}^n$ and note that the isotropy subgroup of $o$ can be identified with $SO(n)$ since it consists of matrices in $SO(n+1)$ of the form 
\begin{equation}
\begin{pmatrix}
1 & 0 \\
0 & R
\end{pmatrix}
\end{equation}
were $R\in SO(n)$. Therefore, we can write $\mathbb{S}^n=SO(n+1)/SO(n)$. It is not possible to define homogeneous cone fields on every $n$-sphere. A direct way of proving this is to note that 
for all
even-dimensional spheres $\mathbb{S}^{2m}$, $m\in\mathbb{N}$ no global cone fields exist. 
This is a clear consequence of the Poincare-Brouwer theorem of algebraic topology, also known as the so-called hairy ball theorem, that the even-dimensional spheres do not admit any globally defined continuous and non-vanishing vector fields. Since a globally defined cone field on a manifold can be used to construct a continuous non-vanishing vector field by a continuous deformation of the cone field to a field of rays, the Poincare-Brouwer theorem implies the non-existence of global cone fields on  $\mathbb{S}^{2m}$.

The set of all $p$-dimensional subspaces of $\mathbb{R}^n$ is called the \emph{Grassmannian} of dimension $p$ in $\mathbb{R}^n$ and is denoted by $\mathrm{Gr}(p,n)$ \cite{Lee2003, Absil2008}. Grassmannians naturally arise in many applications including as parameter spaces in model estimation problems \cite{Smith2005} and in computer vision applications including affine-invariant shape analysis, image matching, and learning theory \cite{Goodall1999,Turaga2008}.
The set $\mathrm{Gr}(p,n)$ can be endowed with a natural differentiable structure that turns it into a compact manifold of dimension $p(n-p)$.
The Grassmann manifold $\mathrm{Gr}(p,n)$ is a homogeneous space with a natural transitive $O(n)$ action \cite{Edelman1999,Besse1987}:
\begin{equation}
\mathrm{Gr}(p,n)=O(n)/\left(O(p)\times O(n-p)\right).
\end{equation}
The Killing form $B:\mathfrak{g\times g}\rightarrow\mathbb{R}$ is non-degenerate for $\mathfrak{g}=\mathfrak{so}(n)$. Thus, $\mathrm{Gr}(p,n)=G/H=O(n)/\left(O(p)\times O(n-p)\right)$ is a reductive homogeneous space with reductive decomposition $\mathfrak{g}=\mathfrak{h}\oplus\mathfrak{m}$, where 
\begin{equation}
\mathfrak{g}=\mathfrak{so}(n), \quad \mathfrak{h}=\bigg\{
\begin{pmatrix}
X_1 & 0 \\
0 &  X_2
\end{pmatrix}:X_1\in \mathfrak{so}(p), \; X_2\in \mathfrak{so}(n-p)\bigg\},
\end{equation}
and $\mathfrak{m}=\mathfrak{h}^{\perp}$ with respect to the Killing form $B$ of $\mathfrak{so}(n)$. That is,
\begin{equation}
\mathfrak{m}=\bigg\{
\begin{pmatrix}
0 & -Z^T \\
Z &  0
\end{pmatrix}:Z\in \mathbb{R}^{(n-p)\times p}\bigg\}.
\end{equation}

As in the case of the $n$-spheres, the Poincare-Hopf theorem can be used to rule out the existence of homogeneous cone fields for most Grassmannians.
Indeed, it can be shown using Schubert calculus \cite{Kleiman1972} that unless $p$ is odd and $n$ is even, the real Grassmannian $\Gr(p,n)$ has a nonzero Euler characteristic and hence does not admit a continuous globally defined cone field as a corollary of the Poincare-Hopf theorem. In particular, $\Gr(p,n)$ does not in general admit a homogeneous cone field. 

\subsubsection{The space of positive definite matrices $S^+_n$}

The space of positive definite matrices $S^+_n$ of dimension $n$ arises in many applications in information geometry and computational science. It is well known that $S^+_n$ is a homogeneous space with a transitive $GL(n)$-action 
given by congruence transformations of the form
\begin{equation}
\tau_A: \Sigma \mapsto A\Sigma A^T \quad \forall A\in GL(n), \ \forall \Sigma \in S^+_n.
\end{equation}
The isotropy group of this action at $\Sigma = I$ is precisely $O(n)$, since $\tau_Q: I \mapsto QIQ^T=I$ if and only if $Q\in O(n)$. Thus, we can identify any $\Sigma\in S^+_n$ with an element of the quotient space $GL(n)/O(n)$. That is
\begin{equation}  \label{quotient}
S^+_n\cong GL(n)/O(n).
\end{equation}

The Lie algebra $\mathfrak{gl}(n)$ of $GL(n)$ consists of the set $\mathbb{R}^{n\times n}$ of all real $n\times n$ matrices
equipped with the Lie bracket $[X,Y]=XY-YX$, while the Lie algebra of $O(n)$ is $\mathfrak{o}(n)=\{X\in \mathbb{R}^{n\times n}: X^T=-X\}$. Since any matrix $X\in \mathbb{R}^{n\times n}$ has a unique decomposition 
$
X=\frac{1}{2}(X-X^T) + \frac{1}{2}(X+X^T)$,
as a sum of an antisymmetric part and a symmetric part, we have $\mathfrak{gl}(n)=\mathfrak{o}(n)\oplus\mathfrak{m}$, where $\mathfrak{m}=\{X\in\mathbb{R}^{n\times n}: X^T=X\}$. Furthermore, since $\operatorname{Ad}_Q(S)=QSQ^{-1}=QSQ^T$ is a symmetric matrix for each $S\in\mathfrak{m}$, we have
$
\operatorname{Ad}_{O(n)}(\mathfrak{m})=\{QSQ^{-1}:Q\in O(n), \ S\in\mathfrak{m}\}\subseteq\mathfrak{m}$.
Hence, $S^+_n=GL(n)/O(n)$ is in fact a reductive homogeneous space with reductive decomposition $\mathfrak{gl}(n)=\mathfrak{o}(n)\oplus\mathfrak{m}$. 
The tangent space $T_oS^+_n$ of $S^+_n$ at the base-point $o=[I]=I\cdot O(n)$ is identified with $\mathfrak{m}$. For each $\Sigma \in S^+_n$, the action $\tau_{\Sigma^{1/2}}:S^+_n\to S^+_n$ induces the vector space isomorphism $d\tau_{\Sigma^{1/2}}\vert_I: T_I S^+_n \to T_{\Sigma} S^+_n$ given by
$d\tau_{\Sigma^{1/2}}\vert_I X = \Sigma^{1/2}X\Sigma^{1/2}$ for each $X\in \mathfrak{m}$, where $\Sigma^{1/2}$ is the unique positive definite square root of $\Sigma$.

A cone field $\mathcal{K}$ on $S^+_n$ is affine-invariant or homogeneous with respect to the quotient geometry $S^+_n\cong GL(n)/O(n)$ if 
\begin{equation}
\left(d\tau_{\Sigma_2^{1/2}\Sigma_1^{-1/2}}\big\vert_{\Sigma_1}\right)\mathcal{K}(\Sigma_1)=\mathcal{K}(\Sigma_2),
\end{equation}
for all $\Sigma_1,\Sigma_2\in S^+_n$. To generate such a cone field, we require a cone $\mathcal{K}(I)\subset\mathfrak{m}$ at identity that is $\Ad_{O(n)}$-invariant: 
\begin{equation}   \label{Ad}
X \in \mathcal{K}(I) \Longleftrightarrow  \operatorname{Ad}_Q X = d\tau_Q\big\vert_I X = QXQ^T \in \mathcal{K}(I), \quad \forall Q\in O(n).
\end{equation}
Using such a cone, we uniquely generate a homogeneous cone field via 
\begin{equation}
\mathcal{K}(\Sigma)= d\tau_{\Sigma^{1/2}}\big\vert_I\mathcal{K}(I) =
\{X\in T_{\Sigma}S^+_n: \Sigma^{-1/2}X\Sigma^{-1/2}\in \mathcal{K}(I)\}.
\end{equation}
The $\operatorname{Ad}_{O(n)}$-invariance condition (\ref{Ad}) is satisfied if $\mathcal{K}(I)$ has a spectral characterization. That is, the characterization of $X\in\mathcal{K}(I)$ must only depend on the spectrum of $X$. For instance, $\operatorname{tr}(X)$ and $\operatorname{tr}(X^2)$ are spectral quantities because $\tr(X)$ is the sum of the eigenvalues of $X$ and $\tr(X^2)$ is the sum of the squares of the eigenvalues of $X$. Therefore, these quantities are $\Ad_{O(n)}$-invariant.
The following result gives a family of quadratic $\Ad_{O(n)}$-invariant cones in $\mathfrak{m}$, each of which generates a distinct homogeneous cone field on $S^+_n$ \cite{Mostajeran2017b}.

\begin{prop} 
For any choice of parameter $\mu \in (0,n)$, the set
\begin{equation}  \label{quad}
\mathcal{K}(I)=\{X\in T_IS^+_n:(\operatorname{tr}(X))^2-\mu \operatorname{tr}(X^2) \geq 0, \ \operatorname{tr}(X) \geq 0\},
\end{equation}
defines an $\operatorname{Ad}_{O(n)}$-invariant cone in $T_{I}S^+_n=\{X\in\mathbb{R}^{n\times n}: X^T=X\}$.
\end{prop}
The parameter $\mu$ controls the opening angle of the cone. If $\mu = 0$, then (\ref{quad}) defines the half-space $\operatorname{tr}(X)\geq 0$. As $\mu$ increases, the opening angle of the cone becomes smaller and for $\mu = n$  (\ref{quad}) collapses to a ray.
Now for any fixed $\mu \in (0,n)$, we obtain a unique well-defined affine-invariant cone field given by
\begin{equation}  \label{quadField}
\mathcal{K}(\Sigma)=\{X\in T_{\Sigma}S^+_n:(\operatorname{tr}(\Sigma^{-1}X))^2-\mu \operatorname{tr}(\Sigma^{-1}X\Sigma^{-1}X) \geq 0, \ \operatorname{tr}(\Sigma^{-1}X) \geq 0\}.
\end{equation}
Of course not all $\operatorname{Ad}_{O(n)}$-invariant cones at $I$ are quadratic.  In particular, the cone of positive semidefinite matrices in $T_{I}S^+_n$ with spectral characterization $\mathcal{K}(I) = \{X\in T_I S^+_n: \lambda_i(X)\geq  0, \; i = 1, \ldots, n\}$ is also $\operatorname{Ad}_{O(n)}$-invariant. The homogeneous cone field generated by this cone at $I$ induces the well-known L\"owner order on $S^+_n$ \cite{Bhatia,Lowner1934}.

\section{Geodesics as conal curves}  \label{4}

A cone field $\mathcal{K}_{\mathcal{M}}$ on a manifold ${\mathcal{M}}$ gives rise to a conal order  $\prec$ on $\mathcal{M}$.
A continuous piecewise smooth curve $\gamma:[t_0,t_1]\rightarrow\mathcal{M}$ is called a conal curve if 
\begin{equation}
\gamma'(t)\in\mathcal{K}_{\mathcal{M}}\left(\gamma(t)\right),
\end{equation}
whenever the derivative exists. For points $a,b\in\mathcal{M}$, we write $a\prec b$ if there exists a conal curve $\gamma:[0,1]\rightarrow\mathcal{M}$ with $\gamma(0)=a$ and $\gamma(1)=b$. If the conal order is also antisymmetric, then it is a partial order. For $x\in\mathcal{M}$, we define the \emph{forward set} $\uparrow x=\{z\in\mathcal{M}:x\prec z\}$ and the \emph{backward set} $\downarrow x=\{z\in\mathcal{M}:z\prec x\}$. In the language of geometric control theory, the forward set of $x$ is called the reachable set from $x$ and the backward set of $x$ is the set controllable to $x$.  The closure $\leq_{\mathcal{K}}$ of this order is again an order and satisfies $x\leq_{\mathcal{K}} y$ if and only if $y\in\overline{\{z:x\prec_{\mathcal{K}} z\}}$. We say that $\mathcal{M}$ is \emph{globally orderable} if $\leq_{\mathcal{K}}$ is a partial order.

The conal order induced by a generic cone field on a path connected manifold is generally highly nontrivial. For instance, given a pair of points $a,b\in \mathcal{M}$, the question of whether $a$ and $b$ are ordered or not is not at all straightforward to answer, since to rule out the existence of an order relation one has to demonstrate that none of an infinite collection of continuous piecewise smooth curves connecting the pair $a,b$ is a conal curve. In this section, we discuss the significant role that geodesics play as conal curves on globally orderable Riemannian homogeneous spaces  with respect to homogeneous cone fields, thereby reducing the search for an order relation between any two points $a,b$ to a single statement on the pair of points. Thus, the use of invariant metric and conal structures on a globally orderable homogeneous space induces an order that is `tractable' in the sense that we can check to see whether two points are ordered by checking a single condition.

First note that if $G$ is a Lie group with a bi-invariant Riemannian metric, then the geodesics in $G$ through the identity element $e$ are precisely the one-parameter subgroups of $G$; i.e., curves $\gamma$ of the form $\gamma(t)=\exp t X$, where $X\in\mathfrak{g}$. That is, the Lie group exponential map coincides with the Riemannian exponential map in such cases. Geodesics through a point $a\in G$ have the form $\gamma(t)=a\cdot \exp(t \, a^{-1}\cdot X)$, where $X\in T_a G$. Now if $G$ is equipped with a left-invariant cone field $\mathcal{K}(g)=dL_{g}\vert_e\mathcal{K}(e)$, then the geodesic $\gamma(t)=a\cdot \exp(t \, a^{-1}\cdot X)$ through $a$ in the direction of $X$ is a conal curve if and only if $X\in\mathcal{K}(a)$, since 
\begin{equation}
\gamma'(t)=a\cdot\exp(t \, a^{-1}\cdot X) \cdot a^{-1}\cdot X \in \mathcal{K}(\gamma(t)) \Longleftrightarrow a^{-1}\cdot X \in \mathcal{K}(e) \Longleftrightarrow X \in \mathcal{K}(a).
\end{equation}
A similar result can be established on homogeneous spaces that are geodesic orbit (g.o.) spaces. A Riemannian manifold $\mathcal{M}=G/H$ is said to be a g.o. space if every geodesic in $\mathcal{M}$ is the orbit of a one-parameter subgroup of $G$. To show that $\mathcal{M}=G/H$ equipped with a homogeneous Riemannian metric is a g.o. space, it is sufficient to show that all geodesics through a single point are orbits of one-parameter subgroups by homogeneity. A Riemannian reductive homogeneous space $G/H$ with reductive decomposition $\mathfrak{g}=\mathfrak{h}\oplus\mathfrak{m}$ is said to be \emph{naturally reductive} if 
$\langle[X,Y]_{\mathfrak{m}},Z\rangle + \langle [X,Z]_{\mathfrak{m}},Y\rangle =0$ for all $X,Y,Z\in \mathfrak{m}$. If $\mathcal{M}=G/H$ is a naturally reductive homogeneous space, then the geodesics of $\mathcal{M}$ through the point $o=eH$ are precisely of the form
\begin{equation}
\gamma(t)=\exp(tX)\cdot o, \quad X\in \mathfrak{m}.
\end{equation}
Furthermore, all symmetric spaces are naturally reductive and thus g.o spaces \cite{Arv2003}.
\begin{prop}
Let $\mathcal{M}=G/H$ be a naturally reductive homogeneous space with reductive decomposition $\mathfrak{g}=\mathfrak{h}\oplus\mathfrak{m}$ that is endowed with a homogeneous Riemannian metric. If $\mathcal{K}$ is a homogeneous cone field on $\mathcal{M}$, then a geodesic $\gamma=\gamma(t)$ through a point $p\in\mathcal{M}$ is a conal curve if and only if $\gamma'(0)\in\mathcal{K}(p)$.
\end{prop}

\begin{proof} 
By homogeneity, it is sufficient to consider geodesics through the base-point $o\in\mathcal{M}$, which are of the form $\gamma(t)=\exp(tX)\cdot o$, $X\in\mathfrak{m}$. Let $\tau_g:\mathcal{M}\rightarrow \mathcal{M}$ be the map $\tau_g(x)=g\cdot x$ for $g\in G$. We have
\begin{equation}
\gamma'(t)=d\tau_{\exp(tX)}\big\vert_{o}X\in\mathcal{K}(\exp(tX)\cdot o)  \Longleftrightarrow X\in\mathcal{K}(o),
\end{equation} 
as $\mathcal{K}$ is homogeneous. That is, $\gamma$ is a conal curve if and only if its initial tangent vector lies in the cone at $o$. See figure \ref{orbit conal}.
\qed
\end{proof}

\begin{figure}
\centering
\includegraphics[width=0.8\linewidth]{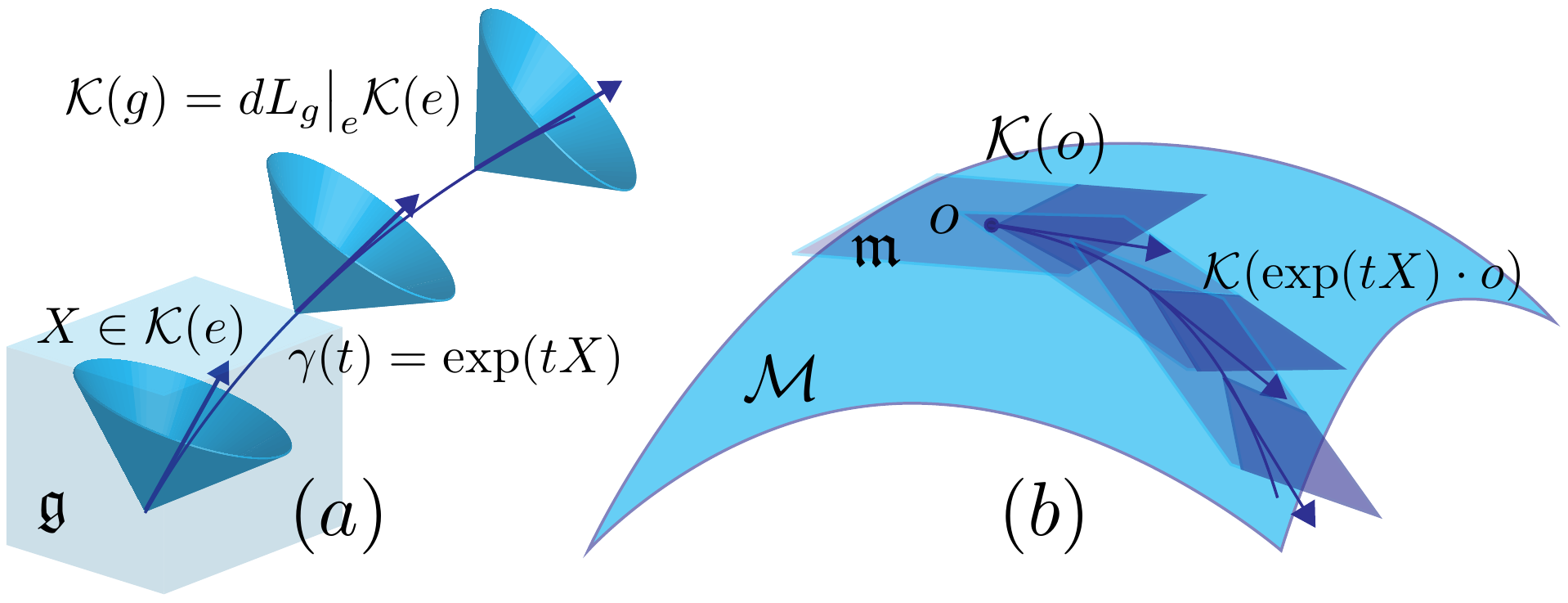}
  \caption{Conal curves arising as orbits of one-parameter subgroups on $(a)$ a Lie group equipped with an invariant cone field, and $(b)$ a naturally reductive homogeneous space with a homogeneous cone field. If the tangent vector lies within the cone at one point, then the whole curve is conal.
  }
  \label{orbit conal}
\end{figure}

We now consider the question of whether two given points $a,b$ on a homogeneous space $\mathcal{M}=G/H$  equipped with a homogeneous cone field $\mathcal{K}$ are ordered. We first treat the example of $\mathbb{R}^n$ endowed with the Euclidean metric and a constant cone field, i.e. invariant with respect to translations. Given $a,b\in\mathbb{R}^n$, we write $a\prec b$ if there exists a curve $\gamma:[0,1]\rightarrow \mathbb{R}^n$ such that $\gamma(0)=a$, $\gamma(1)=b$, and $\gamma'(t)\in\mathcal{K}$, for all $t\in[0,1]$. Since $\mathcal{K}$ is closed and convex, we have
\begin{equation}
\int_0^1\gamma'(t)dt\in \mathcal{K},
\end{equation}
as the integral can be thought of as the limit of a Riemann sum. But, of course, the integral is simply $\gamma(1)-\gamma(0)=b-a$. Thus, to check whether $a$ and $b$ are ordered with respect to a constant cone field, it is sufficient to check $b-a\in\mathcal{K}$; i.e. $a\prec b$ if and only if the straight line from $a$ to $b$ is a conal curve. See figure \ref{real conal}.

\begin{figure}
\centering
\includegraphics[width=0.6\linewidth]{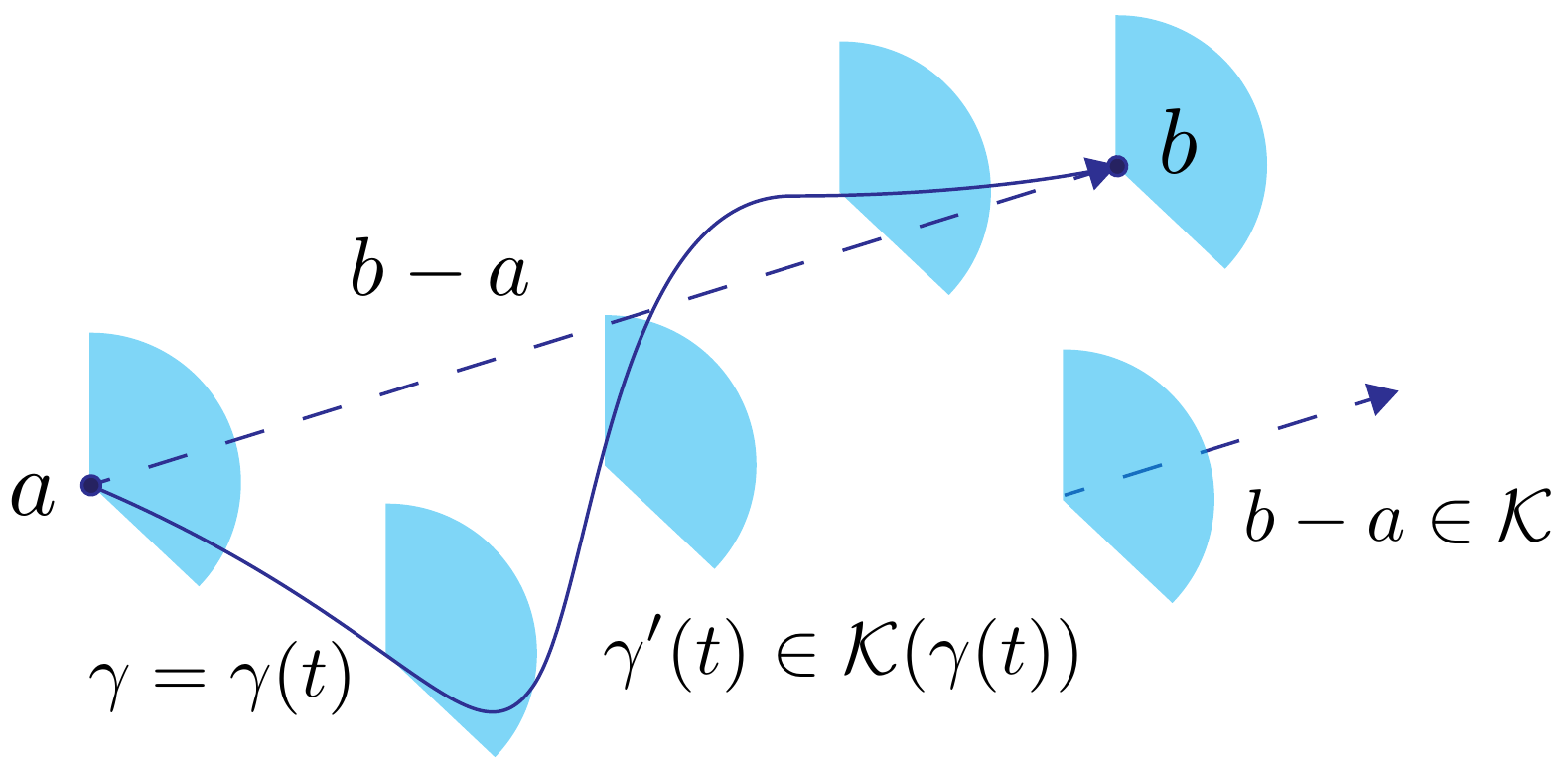}
  \caption{There exists a conal curve joining point $a$ to $b$ in a vector space endowed with a translation invariant cone field if and only if the straight line joining $a$ to $b$ is conal. 
  }
  \label{real conal}
\end{figure}

Now suppose that $\mathcal{M}=G/H$ is a reductive homogeneous space with a global order induced by a homogeneous cone field $\mathcal{K}$. The following theorem is derived from \cite{Neeb1991}.

\begin{thm} \label{Causal semigroup theorem}
Let $\mathcal{K}$ be a homogeneous cone field on $G/H$ arising as the projection of the left-invariant wedge field on $G$ generated by a wedge $W$ as described in Theorem \ref{cone thm}.
If $S=\overline{\langle\exp W\rangle H}\subseteq G$, then
$S=\pi^{-1}\left(\{x\in\mathcal{M}:o\leq_{\mathcal{K}} x\}\right)$ and $G/H$ is globally orderable with respect to $\mathcal{K}$ if and only if $W=\boldsymbol{L}(S)$, where
\begin{equation}
\boldsymbol{L}(S)=\{Z\in\mathfrak{g}:\exp(\mathbb{R}^+Z)\subseteq S\}.
\end{equation}
\end{thm}

It follows from the above result that any element of the set  $\pi^{-1}(\uparrow o)$ in $G$ can be reached as the image of a vector $Z$ in the wedge $W$ under the exponential map. This observation can in turn be used to prove the following result. For simplicity, the theorem is formulated with respect to symmetric spaces, which include all of the examples considered in this paper.

\begin{thm} \label{Causal semigroup theorem 2}
Let $\mathcal{M}=G/H$ be a globally orderable symmetric space with a homogeneous Riemannian metric and homogeneous cone field $\mathcal{K}$. We have $x_1\leq_{\mathcal{K}} x_2$ if and only if the geodesic from $x_1$ to $x_2$ is a conal curve. 
\end{thm}

\begin{proof}
By homogeneity, it is sufficient to consider the case where $x_1=o$ and $x_2=x$. We define a wedge $W$ in $\mathfrak{g}$ by
$
W:=\{Y+X: X\in\mathcal{K}(o), Y\in\mathfrak{h}\}\subset\mathfrak{g}=\mathfrak{h}\oplus\mathfrak{m},
$
where $\mathcal{K}(o)\subset\mathfrak{m}$. Note that $W$ satisfies the relevant properties in Theorem \ref{cone thm} by construction. If $o\leq_{\mathcal{K}} x$, it follows from Theorem \ref{Causal semigroup theorem}  that there exists $Z\in W$ such that 
\begin{equation} \label{pi exp}
x=\pi(\exp Z)=(\exp Z)\cdot o
\end{equation}
By Lawson's polar decomposition theorem \cite{Lawson1991}, any element $g=\exp Z$ of the semigroup $S$ admits a unique decomposition as $g=(\exp X)h$ with $X\in W\cap\mathfrak{m}=\mathcal{K}(o)$ and $h\in H$. Thus, we have
\begin{equation}
x=(\exp Z)\cdot o=(\exp X)h\cdot o=(\exp X) \cdot o,
\end{equation}
since $h\cdot o = o$ for any $h\in H$. Thus, it follows that $\gamma(t)=\exp (tX) \cdot o$ is a conal curve. Since all symmetric spaces are g.o. spaces, $\gamma$ is precisely the Riemannian geodesic from $o$ to $x$.
\qed
\end{proof}

\begin{corollary}
Let $\mathcal{M}$ be a globally orderable symmetric space with metric and conal structures as in Theorem \ref{Causal semigroup theorem 2}. A map $F:\mathcal{M}\rightarrow\mathcal{M}$ is monotone if and only if the geodesic from $F(x_1)$ to $F(x_2)$ is conal whenever the geodesic from $x_1$ to $x_2$ is conal.
\end{corollary}

\subsection{Example: affine-invariant orders on $S^+_n$}

Recall from Section \ref{examples} that $S^+_n$ is a homogeneous space with quotient manifold structure $GL(n)/O(n)$. The exponential map $\exp: \mathfrak{m}\rightarrow S^+_n$ given by the usual matrix power series is a surjective map from the space of $n\times n$ symmetric matrices $\mathfrak{m}$ onto $S^+_n$. Here $\mathfrak{m}$ is identified with the tangent space of $S^+_n$ at the identity $I$. Thus, the logarithm map $\log:S^+_n\rightarrow \mathfrak{m}$ is well-defined on all of $S^+_n$. It is well-known that $S^+_n$ can be equipped with a standard affine-invariant Riemannian metric $\langle \cdot, \cdot \rangle_{\Sigma}$ given by
$\langle X, Y \rangle_{\Sigma}= \tr\left(\Sigma^{-1}X\Sigma^{-1}Y\right)$ for $\Sigma\in S^+_n$, $X,Y\in T_{\Sigma}S^+_n$,
which turns $S^+_n=GL(n)/O(n)$ into a non-compact Riemannian homogeneous space with negative curvature \cite{Lang2012}.
The Riemannian distance between any two points $\Sigma_1,\Sigma_2\in S^+_n$ is given by
\begin{equation} \label{S+ distance}
d(\Sigma_1,\Sigma_2)=\|\log(\Sigma_1^{-\frac{1}{2}}\Sigma_2\Sigma_1^{-\frac{1}{2}})\|_{\mathfrak{m}}=\left(\sum_{i=1}^n\log^2\lambda_i(\Sigma_1^{-1}\Sigma_2)\right)^{1/2},
\end{equation}
where $\lambda_i(\Sigma_1^{-1}\Sigma_2)$, $(i=1, \ldots, n)$ denote the $n$ real and positive eigenvalues of $\Sigma^{-1}_1\Sigma_2$. It follows from (\ref{S+ distance}) that $d(\Sigma^{-1}_1,\Sigma^{-1}_2)=d(\Sigma_1,\Sigma_2)$. Thus, the inversion $\Sigma\mapsto \Sigma^{-1}$ provides an involutive isometry on $S^+_n$, which shows that $S^+_n$ is a Riemannian symmetric space and hence a g.o. space. Therefore, given $\Sigma \in S^+_n$, there exists a unique $X=\log\Sigma\in\mathfrak{m}$, and the geodesic from $I$ to $\Sigma$ is given by $\gamma(t)=\exp(tX)$. Note that the domain of the injective curve $\gamma$ can be extended to all of $\mathbb{R}$. 

In addition to the affine-invariant geometry of $S^+_n$, there is a natural `flat' or translational geometry of $S^+_n$ viewed as a cone embedded in the flat space of $n\times n$ symmetric matrices $\operatorname{Sym}_n$. The L\"owner order $\geq_{L}$ can be defined on $\operatorname{Sym}_n$ by 
\begin{equation}  \label{Lowner 1}
A\geq_{L} B \Longleftrightarrow A-B \geq_{L} O,
\end{equation}
where $A-B \geq_{L} O$ means that $A-B$ is positive semidefinite. The restriction of (\ref{Lowner 1}) to $S^+_n$ defines a partial order on $S^+_n$ that coincides with the order induced on $S^+_n\cong GL(n)/O(n)$ by the homogeneous cone field generated by the cone of positive semidefinite matrices in $\mathfrak{m}\cong T_{I}S^+_n$. That is, in the special case where the cone $\mathcal{K}(I)$ at identity is itself the cone of positive semidefinite matrices, the translation-invariant and affine-invariant cone fields generated by $\mathcal{K}(I)$ agree on $S^+_n$. This is generally not the case for other choices of $\mathcal{K}(I)$, as shown in \cite{Mostajeran2017b}. 

Now let $\mathcal{K}$ be an affine-invariant cone field on $S^+_n$. Such a cone field induces a global partial order on $S^+_n$ \cite{Mostajeran2018}. Given $\Sigma\in S^+_n$, Theorem \ref{Causal semigroup theorem 2}  implies that $I \leq_{\mathcal{K}} \Sigma$ if and only if the geodesic $\gamma(t)=\exp(t\log\Sigma)$ from $I$ to $\Sigma$ is a conal curve; i.e., precisely if $\log\Sigma \in \mathcal{K}(I)$.  By homogeneity, it follows that $\Sigma_1\leq_{\mathcal{K}}\Sigma_2$ if and only if $\log(\Sigma_1^{-1/2}\Sigma_2\Sigma_1^{-1/2})\in\mathcal{K}(I)$. If $\mathcal{K}$ is the cone field corresponding to the  L\"owner order, then $\Sigma_1\leq_{\mathcal{K}}\Sigma_2$ if and only if $\log(\Sigma_1^{-1/2}\Sigma_2\Sigma_1^{-1/2})\geq_{L} O$. If $\mathcal{K}$ is any of the quadratic affine-invariant cone fields in (\ref{quadField}), then $\Sigma_1\leq_{\mathcal{K}}\Sigma_2$ if and only if
\begin{equation} \label{conic 4}
\begin{cases}
\tr\left(\log(\Sigma_1^{-1/2}\Sigma_2\Sigma_2^{-1/2})\right)\geq 0, \\
\left(\tr(\log(\Sigma_1^{-1/2}\Sigma_2\Sigma_1^{-1/2}))\right)^2-\mu\tr\left[(\log(\Sigma_1^{-1/2}\Sigma_2\Sigma_1^{-1/2}))^2\right]\geq 0,
\end{cases}
\end{equation}
which is equivalent to
\begin{equation} \label{conic 6}
\begin{cases}
\sum_i\log\lambda_i\geq 0, \\ \left(\sum_i\log\lambda_i\right)^2-\mu\sum_i(\log\lambda_i)^2\geq 0,
\end{cases}
\end{equation}
where $\lambda_i=\lambda(\Sigma_1^{-1/2}\Sigma_2\Sigma_1^{-1/2})=\lambda_i(\Sigma_2\Sigma_1^{-1})$ $(i=1,...,n)$ denote the $n$ real and positive eigenvalues of $\Sigma_2\Sigma_1^{-1}$. Thus, the use of invariant causal structures has enabled us to answer the question of  whether a pair of positive definite matrices $\Sigma_1$ and $\Sigma_2$ are ordered 
by checking a pair of inequalities involving the spectrum of $\Sigma_2\Sigma_1^{-1}$.

\section{Invariant differential positivity}  \label{6}

\subsection{Cone fields of rank $k$}  

A closed set $\mathcal{C}$ in a vector space $\mathcal{V}$ is said to be a cone of rank $k$ if (i) for any $\lambda\in\mathbb{R}$, $\lambda\cdot\mathcal{C}=\mathcal{C}$, and (ii) the maximum dimension of any subspace of $\mathcal{V}$ contained in $\mathcal{C}$ is $k$. 
Note that if $\mathcal{K}$ is a pointed convex cone, $\mathcal{C}=\mathcal{K}\cup-\mathcal{K}$ is a cone of rank $1$ according to the stated definition. A polyhedral cone of rank $k$ is of the form $\mathcal{C}=\tilde{\mathcal{K}}\cup-\tilde{\mathcal{K}}$, where $\tilde{\mathcal{K}}$ is given as the intersection of a collection of half-spaces. A second class of cones of rank $k$ can be defined using quadratic forms. If $P$ is a symmetric $n\times n$ matrix with $k$ positive eigenvalues and $n-k$ negative eigenvalues, then the set
$
\mathcal{C}=\{x\in\mathcal{V}:\langle x,Px\rangle\geq0\}$,
defines a cone of rank $k$. Note that the closure of the complement of a quadratic cone of rank $k$ in $\mathbb{R}^n$ is a quadratic cone of rank $n-k$. See figure \ref{rank k} for an illustration of polyhedral and quadratic cones of rank 2 in three dimensions.

\begin{figure} 
\centering
\includegraphics[width=0.7\linewidth]{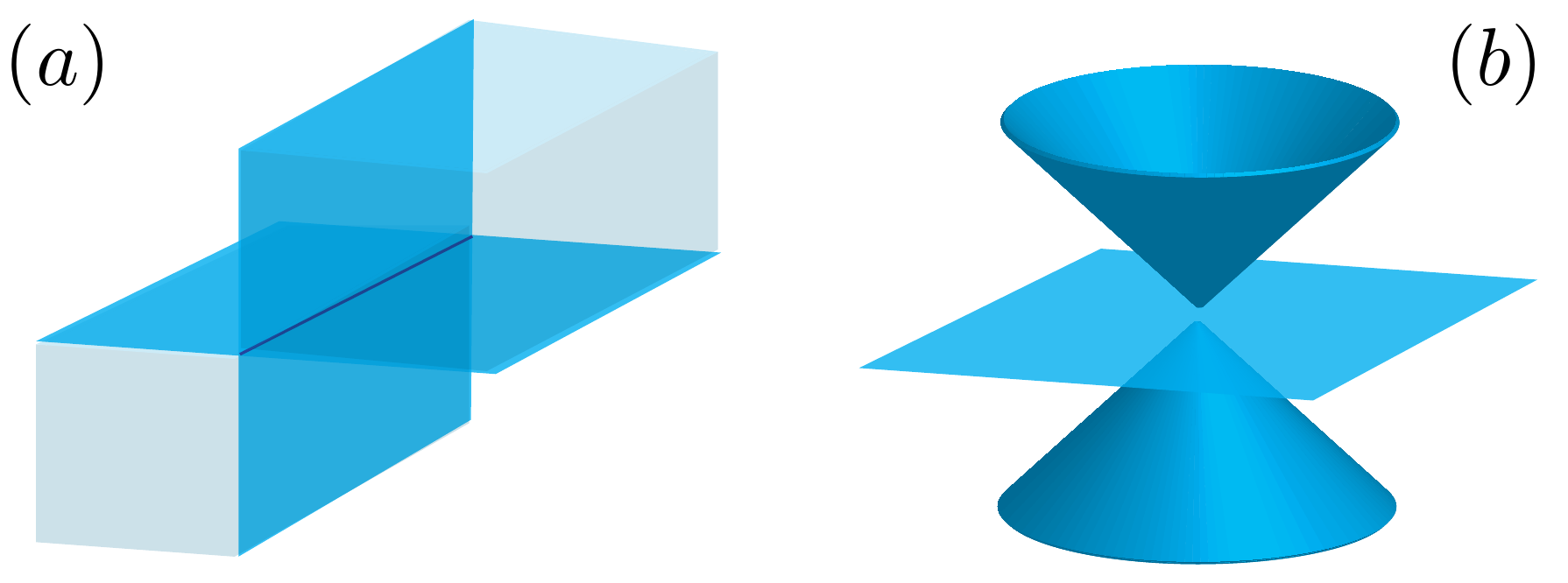}
  \caption{$(a)$ A polyhedral cone of rank $2$ in three dimensions. $(b)$ The double cone is a quadratic cone of rank $1$. The closure of its complement is a cone of rank $2$. Note how a plane can be fitted within this set.
  }
  \label{rank k}
\end{figure}

We can define homogeneous cone fields of rank $k$ on a homogeneous space $\mathcal{M}=G/H$ in an analogous way to pointed and convex homogeneous cone fields. That is, a cone field $\mathcal{C}$ of rank $k$ is homogeneous if it satisfies
\begin{equation}
\mathcal{C}(y)=d\tau_g\big\vert_x\mathcal{C}(x),
\end{equation}
for all $x,y\in \mathcal{M}$ and $g\in G$ such that $y=g\cdot x$. Given a cone $\mathcal{C}(o)\subset T_o\mathcal{M}$ at the base-point $o=eH$, we can extend $\mathcal{C}(o)$ to a unique homogeneous cone field on $\mathcal{M}$ if and only if $\mathcal{C}(o)$ is $\Ad_{H}$-invariant: $\Ad(H)\mathcal{C}(o)=\mathcal{C}(o)$. Given a homogeneous cone field $\mathcal{C}$ of rank $k$ on $\mathcal{M}$, we say that $x_1$ and $x_2$ are related via the cone field and write $x_1\sim x_2$ if there exists a curve $\gamma:[0,1]\rightarrow \mathcal{M}$ such that $\gamma(0)=x_1$, $\gamma(1)=x_2$ and $\gamma'(t)\in \mathcal{C}(\gamma(t))$. 

\subsection{Monotonicity}

A linear map $T:\mathcal{V}\rightarrow\mathcal{V}$ on a vector space $\mathcal{V}$ is positive with respect to a cone $\mathcal{C}$ of any rank $k$ if $T(\mathcal{C})\subseteq\mathcal{C}$. Strict positivity is characterized by $T(\mathcal{C})\subset\operatorname{int} \mathcal{C}$. A smooth map $F:\mathcal{M}\rightarrow\mathcal{M}$ is said to be differentially positive with respect to a cone field $\mathcal{C}=\mathcal{C}(x)$ on $\mathcal{M}$ if $dF\vert_{x}\mathcal{C}(x)\subseteq \mathcal{C}(F(x))$ for every $x\in\mathcal{M}$ \cite{Forni2015}. Invariant differential positivity refers to differential positivity with respect to a homogeneous or invariant cone field on a homogeneous space $\mathcal{M}=G/H$. Since a homogeneous cone field $\mathcal{C}$ satisfies $\mathcal{C}(g\cdot x)=d\tau_{g}\vert_x\mathcal{C}(x)$, where $\tau_g(x)=g\cdot x$ is the left action of $g\in G$ on $x\in\mathcal{M}$, invariant differential positivity of $F$ reduces to
\begin{equation} \label{inv diff}
\left(dF\big\vert_{x}\circ d\tau_{g_1}\big\vert_o\right)\mathcal{C}(o)\subseteq d\tau_{g_2}\big\vert_o\mathcal{C}(o),
\end{equation}
for any $g_1\in\pi^{-1}(x)$ and $g_2\in\pi^{-1}(F(x))$, where $\pi:G\rightarrow\mathcal{M}$ denotes the natural projection map. Note that (\ref{inv diff}) is a condition that is formulated in reference to a single cone $\mathcal{C}(o)\subset T_o\mathcal{M}$.  A continuous-time dynamical system with semiflow $\psi=\psi(t,x)$ is differentially positive if the flow map $\psi_t:\mathcal{M}\rightarrow\mathcal{M}$, $\psi_t(x)=\psi(t,x)$ is differentially positive for any choice of $t>0$.
Notions of strict differential positivity and uniform strict differential positivity are defined in a natural way. In particular, uniform strict differential positivity is characterized by a cone contraction measure that is bounded below by some nonzero factor over a uniform time horizon \cite{Mostajeran2017a}.

Invariant differential positivity with respect to a pointed convex cone field can be thought of as a generalization of monotonicity. Recall that a system is monotone if it preserves a partial order. Monotonicity in a vector space with respect to a constant cone field is equivalent to differential positivity with respect to the same cone field. Similarly, if a homogeneous cone field on a homogeneous space defines a global partial order, then a map on this space is monotone if and only if it is differentially positive with respect to the same conal structure. To see this, note that a smooth map $F:\mathcal{M}\rightarrow\mathcal{M}$ is monotone with respect to a partial order $\geq$ on $\mathcal{M}$ if $F(x_1)\geq F(x_2)$ whenever $x_1\geq x_2$. Let $\geq_{\mathcal{K}}$ denote the partial order induced by a homogeneous cone field $\mathcal{K}$ on $\mathcal{M}$. If $x_1\geq_{\mathcal{K}}x_2$, then there exists a conal curve 
$\gamma:[0,1]\to \mathcal{M}$ such that $\gamma(0)=x_ 1$, $\gamma(1)=x_2$ and $\gamma'(t)\in \mathcal{K}(\gamma(t))$ for all $t\in (0,1)$. Now $F\circ\gamma: [0,1]\to \mathcal{M}$ is a curve in $\mathcal{M}$ with $(F\circ\gamma) (0)=F(x_1)$, $(F\circ\gamma)(1)=F(x_2)$, and
$
(F\circ\gamma)'(t)=dF\vert_{\gamma(t)}\gamma'(t)$.
Thus, $F\circ\gamma$ is a conal curve joining $F(x_1)$ to $F(x_2)$ if and only if 
$dF\vert_{\gamma(t)}\mathcal{K}(\gamma(t))\subseteq\mathcal{K}(F(\gamma(t))$; i.e., $F$ is differentially positive with respect to $\mathcal{K}$ as expected. 
If an invariant cone field on a homogeneous space does not induce a partial order due to a failure of the antisymmetry condition arising from topological constraints, then monotonicity is not defined as it relies on the existence of a partial order. Nonetheless, invariant differential positivity provides the natural extension of the concept of monotonicity in this setting. 

\subsubsection{Example: Order-preserving maps on $S^+_n$}

Invariant differential positivity can be a powerful tool for establishing monotonicity when the cone field characterization of a partial order is available. It can also be a particularly effective tool in proving monotonicity with respect to a family of partial orders corresponding to a collection of causal structures.
A fundamental result in operator theory is the L\"owner-Heinz theorem \cite{Bhatia, Lowner1934}, which states that the map
$\Sigma\mapsto \Sigma^r$ on $S^+_n$ $(n\geq 2)$ is monotone with respect to the L\"owner order if and only if $r\in[0,1]$. The following theorem provides an extension of this important result to an infinite collection of affine-invariant causal structures on $S^+_n$, thereby highlighting the intimate connection of the L\"owner-Heinz theorem to the affine-invariant geometry of $S^+_n$. The proof is based on differential positivity and can be found in \cite{Mostajeran2018}. 

\begin{thm} \label{lowner general}
Let $\geq_{\mathcal{K}}$ denote the partial order induced by any quadratic affine-invariant cone field $\mathcal{K}$ on $S^+_n$. If $\Sigma_1\geq_{\mathcal{K}} \Sigma_2$ in $S^+_n$ and $r\in[0,1]$, then 
\begin{equation}
\Sigma_1^r\geq_{\mathcal{K}}\Sigma_2^r.
\end{equation}
Furthermore, if $n\geq 2$ and $r>1$, then the map $\Sigma\mapsto\Sigma^r$ is not monotone with respect to $\geq_{\mathcal{K}}$.
\end{thm}

\subsection{Strict positivity}

In linear positivity theory, strict positivity of a system with respect to a cone of rank $k$ implies the existence of a dominant eigenspace of dimension $k$, which is an attractor for the system \cite{Fusco1991}. In the differential theory, the notion of a dominant eigenspace is replaced with that of a forward invariant distribution $\mathcal{D}$ of rank $k$ corresponding to the dominant modes of the linearized system \cite{Mostajeran2017d,Forni2017dominance}. It is this distribution that shapes the asymptotic behavior of the dynamics. In particular, if $\mathcal{D}$ is involutive in the sense that for any pair of smooth vector fields $X,Y$ defined near $x\in\mathcal{M}$, $X(x), Y(x)\in \mathcal{D}_x\subset T_x\mathcal{M}\Rightarrow [X,Y](x)\in\mathcal{D}_x$, then an integral manifold of $\mathcal{D}$ is an attractor of the system under suitable technical conditions.

\begin{thm} \label{rank thm}
Let $\Sigma$ be a uniformly strictly differentially positive system with respect to an invariant cone field $\mathcal{C}$ of rank $k$ on a homogeneous Riemannian manifold $\mathcal{M}$ in a bounded, connected and forward-invariant region $S\subseteq \mathcal{M}$. If the forward-invariant distribution of rank $k$ corresponding to the $k$ dominant modes of linearizations 
of $\Sigma$ is involutive and satisfies 
\begin{equation} \label{technical 1}
\limsup_{t\rightarrow\infty}\|d\psi_t\vert_x w\|_{\psi_t(x)}<\infty,
\end{equation}
for all $w\in\mathcal{D}_x$, then there exists a unique integral manifold of $\mathcal{D}$ that is an attractor for all the trajectories from $S$.
\end{thm}

\begin{proof}
The strict differential positivity of $\Sigma$ with respect to $\mathcal{C}$ on a bounded forward invariant region determines a splitting $T_x\mathcal{M}=\mathcal{D}_x\oplus\mathcal{D}'_x$, where $\mathcal{D}_x$ and $\mathcal{D}'_x$ are distributions of rank $k$ and $n-k$, respectively, and $\mathcal{D}_x$ is forward-invariant:
\begin{equation}
d\psi_t\big\vert_x\mathcal{D}_x\subseteq\mathcal{D}_{\psi_t(x)}, \quad \forall x \in \mathcal{M}, \; \forall t>0,
\end{equation}
and corresponds to the $k$ dominant modes of the linearized system by arguments that can be found in the proofs of Theorem 1.2 of \cite{Newhouse2004} and Theorem 1 of \cite{Forni2017dominance}. For each $x\in \mathcal{M}$, define the map $\Phi_x:\mathcal{C}(x)\setminus \{0_x\}\rightarrow\mathbb{R}^{\geq 0}$ by
\begin{equation}
\Phi_x(\delta x) =\frac{\|\Pi_{\mathcal{D}'_x}\delta x\|_x}{\|\Pi_{\mathcal{D}_x}\delta x\|_x},
\end{equation}
where $\Pi_{\mathcal{D}_x}$, $\Pi_{\mathcal{D}'_x}$ denote the linear projections onto the subspaces $\mathcal{D}_x$ and $\mathcal{D}'_x$. The map $\Phi(x)$ is clearly well-defined since $\Pi_{\mathcal{D}_x}\delta x\neq 0_x$ for any $\delta x\in \mathcal{C}(x)\setminus \{0_x\}$. Strict differential positivity ensures that 
$
\lim_{t\rightarrow\infty}\Phi_{\psi_t(x)}\left(d\psi_t\big\vert_x\delta x\right)=0$
for all $x\in \mathcal{M}$ and $\delta x \in \mathcal{C}(x)\setminus \{0_x\}$.

Now if $w \in \mathcal{D}_x$, (\ref{technical 1}) guarantees that
$\lim_{t\rightarrow\infty}\Phi_{\psi_t(x)}\left(d\psi_t\big\vert_x\delta x\right)=0$ implies that $\lim_{t\rightarrow \infty}\Pi_{\mathcal{D}'_{\psi_t(x)}}(d\psi_t\vert_x\delta x)=0$. 
If $\delta x\notin \mathcal{C}(x)$, then for some $\alpha >0$ and $w\in\mathcal{D}_x$, we have $\delta x + \alpha w\in \mathcal{C}(x)\setminus\{0_x\}$ and $\lim_{t\rightarrow\infty}\Phi_{\psi_t(x)}\left(\delta x+\alpha w\right)=0$, which implies that $\lim_{t\rightarrow \infty}\Pi_{\mathcal{D}'_{\psi_t(x)}}(d\psi_t\vert_x\delta x)=0$ once again. Thus, in the limit of $t\rightarrow \infty$, $\delta x(t) = d\psi_t\vert_x\delta x$ becomes parallel to $\mathcal{D}_{\psi_t(x)}$. It follows that any curve $\Gamma=\Gamma(s)$ in $S$ evolves so that $\psi_t(\Gamma(s))$ asymptotically lies on an integral manifold of $\mathcal{D}$.
To prove uniqueness of the attractor, assume for contradiction that $\mathcal{N}_1$ and $\mathcal{N}_2$ are two distinct attractive integral manifolds of $\mathcal{D}$ and let $x_1\in \mathcal{N}_1$, $x_2\in\mathcal{N}_2$. By connectedness of $S$, there exists a smooth curve $\Gamma$ in $S$ connecting $x_1$ and $x_2$. Since the curve $\psi_t(\Gamma(s))$ converges to an integral manifold of $\mathcal{D}$, $\mathcal{N}_1$ and $\mathcal{N}_2$ must be subsets of the same integral manifold of $\mathcal{D}$, which provides the contradiction that completes the proof.
\qed
\end{proof}

Note that condition (\ref{technical 1}) is necessary to ensure that vectors along the distribution $\mathcal{D}$ do not grow unbounded as they evolve by the variational flow, thereby ensuring that strict differential positivity results in contraction toward an integral manifold of $\mathcal{D}$.

As noted earlier, invariant differential positivity is a generalization of monotonicity to homogeneous spaces.
This generalization is made possible by the nonlinearity of the homogeneous space and allows for more complex asymptotic behavior to arise. For instance, as shown in the work of M. Hirsch \cite{Hirsch1988,Smith1995}, almost all bounded trajectories of a strongly monotone (strictly differentially positive) system converge to the set of equilibria. Moreover, under mild smoothness and boundedness assumptions, almost every trajectory converges to one equilibrium. On the other hand, for systems that are strictly differentially positive with respect to pointed convex invariant cone fields on homogeneous spaces, almost all bounded trajectories may converge to a limit cycle under similarly mild technical assumptions \cite{Forni2015}. For example, invariant differential positivity on the cylinder has been used to establish convergence to a unique limit cycle in a nonlinear pendulum model \cite{Forni2014a}. Such a one-dimensional asymptotic behavior is possible for an invariantly differentially positive system due to the topology of the cylinder, which allows for the existence of closed conal curves.

\subsection{Coset stabilization on Lie groups}

In many problems in dynamical systems and control theory, we are interested in asymptotic convergence of trajectories to a submanifold of the state space. Such problems arise in numerous applications including consensus, synchronization, pattern generation, and path following. Here we consider a special class of such systems which are defined on a Lie group $G$ and converge to submanifolds which arise as integrals of left-invariant distributions on $G$. A left-invariant distribution $\mathcal{D}$ on $G$ is a distribution that satisfies  $\mathcal{D}_{g_1g_2}=dL_{g_1}\vert_{g_2}\mathcal{D}_{g_2}$ for all $g_1,g_2\in G$. Any such distribution uniquely determines a subspace $\mathcal{D}_e$ of $T_eG\cong\mathfrak{g}$ and conversely every subspace of $T_eG$ defines a unique left-invariant distribution. Furthermore, the left-invariant distribution defined by a subspace $U$ of the Lie algebra $\mathfrak{g}$ is integrable if and only if $U$ is a subalgebra $\mathfrak{h}$ of $\mathfrak{g}$. Given such a left-invariant distribution $\mathcal{D}$, its integral through the identity element $e\in G$ is a subgroup $H$ of $G$ with Lie algebra $\mathfrak{h}$. The integral of $\mathcal{D}$ through any other point $g\in G$ corresponds to a translation of the subgroup $H$ on $G$ and can be identified with the left coset $gH$. 

In many applications, we are interested in stabilizing a submanifold corresponding to a coset $gH$ in $G$. For instance, in satellite surveillance the attitude of the satellite is an element of the special orthogonal group $SO(3)$ and we often seek to control the orientation of the satellite by ensuring that the telescope axis points to a fixed point on the Earth's surface. Since the set of attitudes which solve this problem can be identified with an $SO(2)$ subgroup corresponding to rotations about the telescope axis, the problem is essentially a simple coset stabilization problem \cite{Montenbruck2016}. Another large class of coset stabilization problems involves consensus models involving $N$ agents $g_k$ whose states evolve on a Lie group $G$ and exchange information about their relative positions via a communication graph. The consensus manifolds in such problems take the form of a single copy of $G$ repeated $N$ times and diagonally embedded in the Cartesian product of $G^N$, which is the $(N\dim G)$-dimensional state space of the system. Such consensus manifolds correspond to fixed formations of the $N$ agents evolving uniformly on $G$ and correspond to $(\dim G)$-dimensional cosets in $G^N$. The synchronization manifold on which all $N$ agents have the same state is a special case of such a coset in $G^N$.
A simple example of such a model is a network of $N$ oscillators evolving on the $N$-torus, where the cosets are one-dimensional and correspond to frequency synchronization and phase-locking behaviors. 

Now consider a homogeneous space $\mathcal{M}=G/H$, where $G$ is a semisimple and compact Lie group so that the negative of the Killing form $-B$  of $\mathfrak{g}$ is positive definite. Then $\mathcal{M}$ is reductive 
with reductive decomposition $\mathfrak{g}=\mathfrak{h}\oplus\mathfrak{m}$, where $\mathfrak{m}=\mathfrak{h}^{\perp}$. The negative of the Killing form $-B$ induces a bi-invariant metric on $G$ and a homogeneous metric on $\mathcal{M}$. The tangent space at each point $g\in G$ admits the decomposition $T_gG=\mathcal{H}_g\oplus \mathcal{H}_g^{\perp}$, where $\mathcal{H}=dL_g\vert_e\mathfrak{h}$ and $\mathcal{H}^{\perp}=dL_g\vert_e\mathfrak{m}$.
We choose a basis for $\mathfrak{g}$ of the form
$\{e_1,...,e_{\dim\mathfrak{h}},e_{\dim\mathfrak{h}+1},...,e_{\dim\mathfrak{g}}\}$,
where $\{e_i: i=1,...,\dim \mathfrak{h}\}$ is a basis of $\mathfrak{h}$. Define a cone $\mathcal{C}(e)$ of rank $k=\dim\mathfrak{h}$ in $\mathfrak{g}$ by
\begin{equation}  \label{cone lie}
\mathcal{C}(e)=\bigg\{v=\sum_{i=1}^{\dim\mathfrak{g}}v_ie_i\in\mathfrak{g}:
Q(v):=\sum_{i=1}^{\dim\mathfrak{h}}v_i^2-\mu\sum_{i=\dim\mathfrak{h}+1}^{\dim\mathfrak{g}}v_i^2\geq 0
\bigg\},
\end{equation}
where $\mu>0$ is a sufficiently small parameter to ensure that (\ref{cone lie}) is non-empty. $\mathcal{C}(e)$ can be uniquely extended to a left-invariant cone field $\mathcal{C}(g)=dL_g\vert_e\mathcal{C}(e)$ of rank $\dim\mathfrak{h}$ on $G$.

Note that a dynamical system $\dot{g}=\bar{f}(g)$ on $G$ induces a well-defined projected flow on $\mathcal{M}$ generated by $\dot{x}=f(x)$ if 
\begin{equation}  \label{horizontal}
\bar{f}_{\mathcal{H}^{\perp}}(g_1)=\bar{f}_{\mathcal{H}^{\perp}}(g_2)
\end{equation}
for all $g_1\sim g_2$, where $\sim$ denotes the equivalence relation induced by the coset manifold structure $G/H$; i.e., $g_1\sim g_2$ if and only if there exists $h\in H$ such that $g_1=g_2h$. $\bar{f}_{\mathcal{H}^{\perp}}(g)$ in (\ref{horizontal}) refers to the component of $\bar{f}$ in $\mathcal{H}^{\perp}_g\subset T_gG$. The following theorem shows how strict differential positivity with respect to the left-invariant cone field $\mathcal{C}$ in $G$ can induce a contractive dynamical system on $\mathcal{M}$ under suitable technical conditions. See figures \ref{cycle} and \ref{invariant} for illustrations of the relevant concepts.

\begin{thm} \label{proj thm}
Let $\Sigma: \dot{g}=\bar{f}(g)$ be a system on $G$ that induces a well-defined projective dynamics $\dot{x}=f(x)$ on $\mathcal{M}=G/H$. Suppose that $\Sigma$ is uniformly strictly differentially positive with respect to the left-invariant cone field $\mathcal{C}$ determined by (\ref{cone lie}) on a bounded, connected and forward invariant region $S\subset G$, with a distribution of dominant eigenspaces of the form $\mathcal{D}_g=dL_g\vert_e\mathfrak{h}$. If $\limsup_{t\rightarrow\infty}\|d\psi_t\vert_g w\|_{\psi_t(g)}<\infty$,
for all $w\in\mathcal{D}_g$, then there exists a unique coset $gH$ in $S$ that is an attractor for all trajectories from $S$. Furthermore, the induced  system on $\mathcal{M}$ is contractive with respect to any homogeneous metric and all trajectories from $\pi(S)\subset\mathcal{M}$ converge to the fixed point $x=\pi(gH)\in\pi(S)$.
\end{thm}

\begin{figure}
\centering
\includegraphics[width=0.8\linewidth]{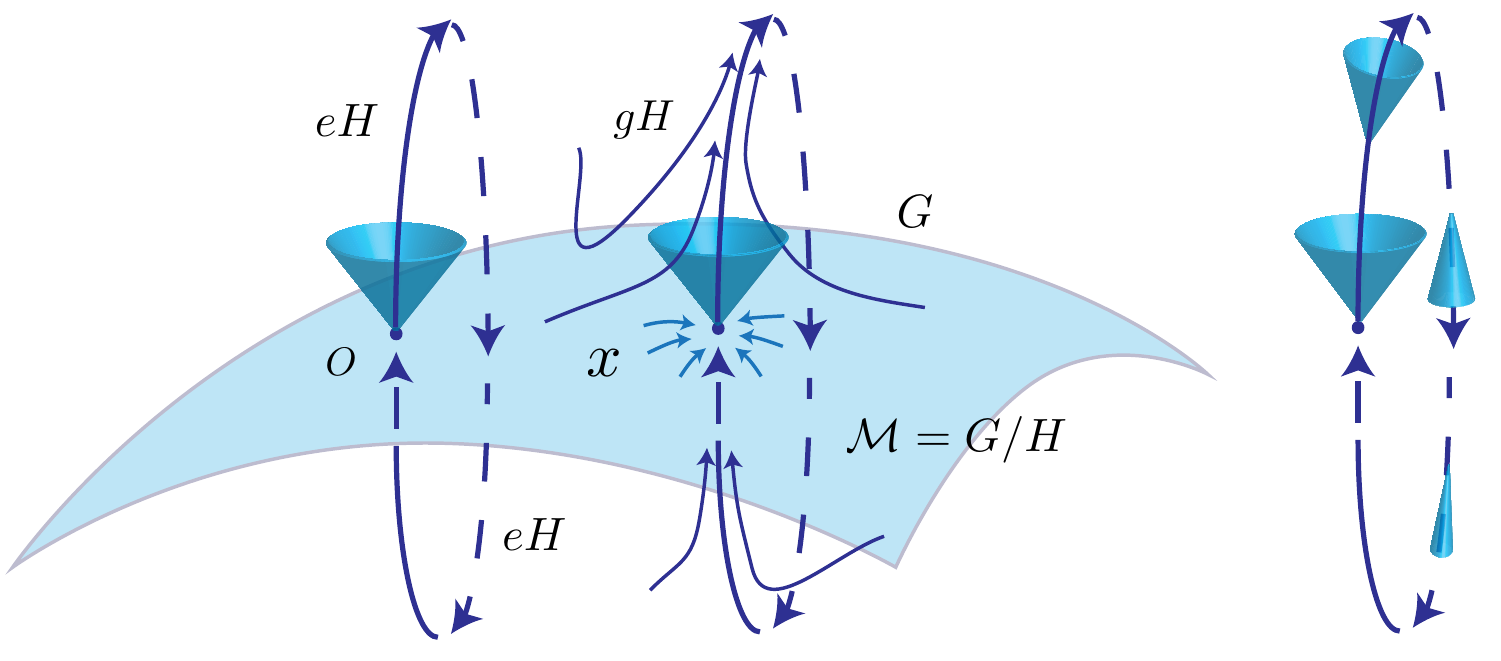}
  \caption{An invariantly strictly differentially positive system on a Lie group $G$ inducing a contractive system on $\mathcal{M}=G/H$. Under suitable technical conditions, trajectories in $G$ converge to a coset $gH$ in $G$ corresponding to a point $x\in\mathcal{M}$. In this figure the cosets are depicted as circles $\mathbb{S}^1$.
  }
  \label{cycle}
\end{figure}

\begin{figure}
\centering
\includegraphics[width=0.9\linewidth]{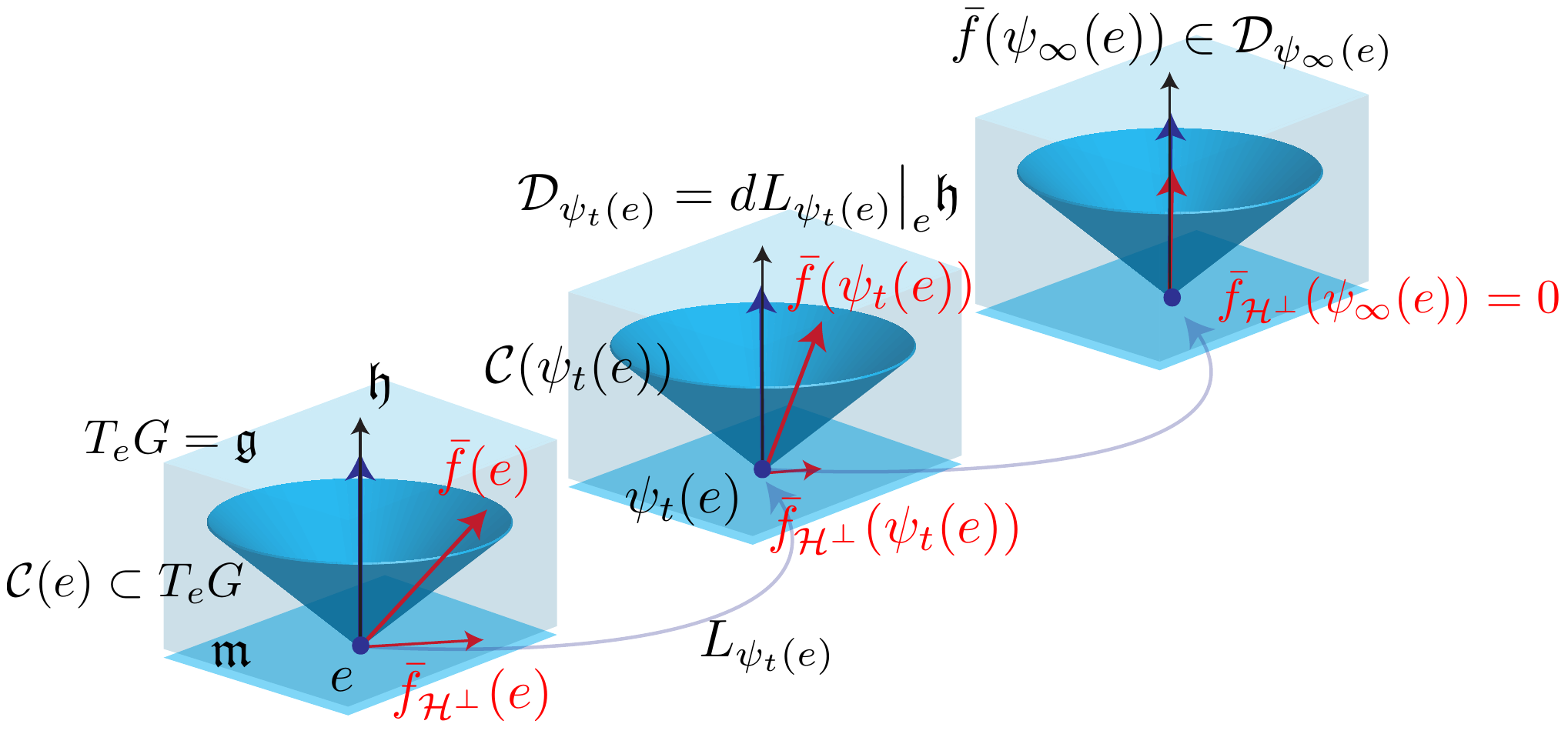}
  \caption{Trajectories of an invariantly strictly differentially positive system asymptotically aligning with a distribution $\mathcal{D}$ of dominant eigenspaces of the linearized system. In this figure the distribution corresponds to the left-invariant distribution generated by a subalgebra $\mathfrak{h}$ of $\mathfrak{g}$.
  }
  \label{invariant}
\end{figure}

\subsubsection{Example: Consensus on $\mathbb{S}^1$}

Let $G$ be a compact Lie group with a bi-invariant Riemannian metric giving rise to a distance function $d:G \times G\rightarrow [0,\infty)$. Given a network of $N$ agents $g_k$ represented by an undirected connected graph $\mathcal{G}=(\mathcal{V},\mathcal{E})$ consisting of a set of vertices $\mathcal{V}$ and edges $\mathcal{E}$ evolving on $G$, we can define a class of consensus protocols on $G$ as follows.  For each $g_k\in G$ denote the Riemannian exponential and logarithm maps by $\exp_{g_k}:T_{g_k}G\rightarrow G$ and $\log_{g_k}:U_{g_k}\rightarrow T_{g_k}G$, respectively, where $U_{g_k}\subset G$ is the maximal set containing $g_k$ for which $\exp_{g_k}$ is a diffeomorphism. The system given by
\begin{equation} \label{consensus protocol}
\dot{g}_k=g_k\cdot\Omega_k+\sum_{i:(k,i)\in\mathcal{E}}\mu_{ki}(d(g_k,g_i))\frac{\log_{g_k}g_i}{\|\log_{g_k}g_i\|},
\end{equation}
defines a consensus protocol on $G$ for constant vectors $\Omega_k\in\mathfrak{g}$ and any collection of real-valued  reshaping functions $\mu_{ki}$ of the distance $d$ that satisfy $\mu_{ki}(0)=0$. Equation (\ref{consensus protocol}) defines a dynamical system on $G\times \cdots \times G = G^N$ that yields a well-defined projected system on $G^N/G$ in the sense of Eq. (\ref{horizontal}) and Theorem \ref{proj thm}, since 
\begin{equation}
\log_{g\cdot g_k}(g\cdot g_i)= g\cdot (\log_{g_k}g_i) \quad \mathrm{and} \quad d(g\cdot g_k,g\cdot g_i)=d(g_k,g_i),
\end{equation}
for all $g,g_k,g_i\in G$ leave Eq. (\ref{consensus protocol}) invariant.

For the sake of simplicity, we will consider the example of a network of $N$ agents evolving on the circle $\mathbb{S}^1$. Equation (\ref{consensus protocol}) reduces to a system of the form
\begin{equation} \label{circle consensus}
\dot{\theta}_k=\omega_k +\sum_{i:(k,i)\in\mathcal{E}}\mu_{ki}(\theta_k-\theta_i),
\end{equation}
where $\theta_k\in\mathbb{S}^1$ represents the phase of agent $k$, $\omega_k\in\mathbb{R}$ are prescribed `intrinsic' frequencies, and $\mu_{ki}$ now denotes an odd coupling function on the domain $(-\pi,\pi)$ extended to $\mathbb{R}$ in such a way so as to make it $2\pi$-periodic. Note that $\mu_{ki}$ and $\mu_{ik}$ need not be the same function. Let $\theta=(\theta_1,\ldots,\theta_N)$ denote an element of the $N$-torus $\mathbb{T}^N$ and consider the $N$-tuple of vector fields $
\left(\frac{\partial}{\partial\theta^1},\ldots,\frac{\partial}{\partial\theta^N}\right)
$, which defines a basis of left-invariant vector fields on $\mathbb{T}^N$.
Assuming that the coupling functions $\mu_{ki}$ are differentiable and strictly monotonically increasing on $(-\pi,\pi)$, then it can be shown that the linearization $\dot{\delta\theta} = A(\theta)\delta\theta$ of the system given by (\ref{circle consensus}) is uniformly strictly differentially positive on the set $\mathbb{T}^N_{\pi}=\{\vartheta\in\mathbb{T}^N:|\vartheta_k-\vartheta_i|<\pi,\;(i,k)\in\mathcal{E}\}$ with respect to the invariant cone field
\begin{equation}
\mathcal{K}_{\mathbb{T}^N}(\theta):=\bigg\{\delta \theta\in T_{\theta}\mathbb{T}^N: \delta \theta^i\geq 0, \, \delta\theta=\sum_i\delta\theta^i\frac{\partial}{\partial\theta^i}\bigg\},
\end{equation}
 for any strongly connected communication graph. Furthermore, the Perron-Frobenius vector field of the system on $\mathbb{T}^N_{\pi}$ is the left-invariant vector field $\boldsymbol{1}(\theta)=(1,\ldots,1)\in T_{\theta}\mathbb{T}^{N}$, where the vector representation is given with respect to the invariant basis defined by $\left(\frac{\partial}{\partial\theta^1},\ldots,\frac{\partial}{\partial\theta^N}\right)$. Moreover, if we denote the flow of (\ref{circle consensus}) by $\psi_t$, then 
 the condition $A(\theta)\boldsymbol{1}(\theta)=0$ implies that $d\psi_t\vert_{\theta}\boldsymbol{1}_{\theta}=\boldsymbol{1}_{\psi_t(\theta)}$, which ensures that $\limsup_{t\rightarrow\infty}\|d\psi_t\vert_{\theta} \boldsymbol{1}(\theta)\|_{\psi_t(\theta)}<\infty$ for any flow confined to $\mathbb{T}^N_{\pi}$, where $\|\cdot\|_{\theta}$ denotes the norm corresponding to the standard Riemannian metric on $\mathbb{T}^N$. If we add the requirement that the coupling functions $\mu_{ki}$ be barrier functions on $(-\pi,\pi)$ so that $\mu_{ki}(\alpha)\rightarrow \infty$ as $\alpha\rightarrow \pi$, then the flow $\psi_t$ will be forward-invariant on $\mathbb{T}^N_{\pi}$, resulting in the following theorem.

\begin{thm}
Consider a network of agents on $\mathbb{S}^1$ communicating via a strongly connected communication graph according to (\ref{circle consensus}). If the coupling functions $\mu_{ki}$ satisfy $\mu_{ki}(0)=0$, $\mu_{ki}(\alpha)\rightarrow \infty$ as $\alpha \rightarrow \pi$, and $\mu_{ki}'(\alpha)>0$ on $(-\pi,\pi)$, then every trajectory from $\mathbb{T}_{\pi}^N$ converges to an integral curve of the vector field $\boldsymbol{1}=\boldsymbol{1}(\theta)$.
\end{thm}

Note that convergence to an integral curve of $\boldsymbol{1}$ on $\mathbb{T}^N$ corresponds to a phase-locking behavior, whereby the collective motion asymptotically converges to movement in a fixed formation with frequency synchronization among the agents. Further details may be found in \cite{Mostajeran2017a}.

\section{Conclusion}

We have reviewed the notion of invariant cone fields on homogeneous spaces and presented examples of cone fields on a list of homogeneous spaces that are of special interest in applications in information science. Invariant differential positivity naturally arises as a generalization of monotonicity on homogeneous spaces within this context. Finally, we have illustrated the potential applications of monotone flows on homogeneous spaces in systems and control theory by reviewing the consensus problem on $\mathbb{S}^1$ and discussing possible extensions of the approach to higher-dimensional spaces.

\begin{acknowledgements}
The authors are grateful to Dr. Fulvio Forni for numerous stimulating conversations and insightful comments on this work.
\end{acknowledgements}


\bibliographystyle{spbasic}      

\bibliography{references}

\end{document}